\documentclass[12pt]{article}
\usepackage{mathrsfs}
\usepackage{}

\usepackage{authblk}
\usepackage{cite}
\usepackage{amsmath}
\usepackage{amsthm,amsfonts,amssymb}           
\usepackage{bbm}
\usepackage[left=2.5cm,
            right=2.5cm
	       ]{geometry}
\usepackage[pdfstartview=FitH,
            CJKbookmarks=true,
            bookmarksnumbered=true,
            bookmarksopen=true,
            colorlinks,
            linkcolor=blue,
            anchorcolor=blue,
            citecolor=blue,
            urlcolor=blue
            ]{hyperref}

\newtheorem*{lem2*}{Lemma 1.2*}
\newtheorem*{lem1*}{Lemma 1.1*}

\newtheorem{mthm}{Theorem}
\newtheorem{mcor}[mthm]{Corollary}

\newtheorem{thm}[mthm]{Theorem}
\newtheorem{cor}[mthm]{Corollary}
\newtheorem{lem}[mthm]{Lemma}

\theoremstyle{remark}

\def\R{{\mathbb R}}
\def\Q{{\mathbb Q}}

\newcommand*{\dif}{\mathop{}\!\mathrm{d}}             
\def\lin{\operatorname{lin}}

\def\relint{\operatorname{relint}}

\newcommand{\GL}[2]{\mathrm{GL}_{#1}(#2)}                 
\newcommand{\SLn}{\mathrm{SL}(n)}                  
\newcommand{\sln}{\mathrm{SL}(n)}                  
\newcommand{\gln}{\mathrm{GL}(n)}                  
\newcommand{\Sn}{S^{n-1}}                  

\newcommand{\kik}[1]{\mathbbm{1}_{#1}}                  
\def\MP{\mathcal{P}}
\def\MK{\mathcal{K}}
\def\MS{\mathcal{S}}
\def\MA{\mathcal{A}}

\def\MT{\mathcal{T}}
\newcommand\CF[1]{C(\R^{#1})}                                    
\newcommand\CFr[1]{C_r(\R^{#1}\setminus \{o\})}                  
\newcommand\CFo[1]{C(\R^{#1}\setminus \{o\})}                  
\newcommand{\ACF}[1]{AC_r(\R^{#1}\setminus \{o\})}        
\newcommand{\ro}[1]{\R^{#1}\setminus \{o\}}                

\def\MF{\mathcal{F}}
\newcommand{\FF}[1]{F(\R^{#1})}                     
\newcommand{\FFo}[1]{F(\R^{#1}\setminus \{o\})}     
\newcommand{\CSF}[1]{C^\infty(\R^{#1})}                   
\newcommand{\cms}{\mathcal{M}^c(\R)}                   

\newcommand\mpo[1]{\MP_{o}^{#1}}

\def\MTon{\MT_o^n}
\newcommand\ab[1]{\left(#1\right)}        

\newcommand\eutr[2]{E^{#1}_{#2}}       
\newcommand\face[2]{\MF^{#1}_{#2}}
\def\a{\alpha}

\def\e{\epsilon}

\def\lt{\mathcal{L}}

\makeatletter
\newcommand{\subjclass}[2][1991]{%
  \let\@oldtitle\@title%
  \gdef\@title{\@oldtitle\footnotetext{#1 \emph{Mathematics subject classification.} #2}}%
}
\newcommand{\keywords}[1]{%
  \let\@@oldtitle\@title%
  \gdef\@title{\@@oldtitle\footnotetext{\emph{Key words and phrases.} #1.}}%
}
\makeatother

\title{\bf{$\sln$ covariant function-valued valuations}}
\author{Jin Li}
\affil{Institut f\"{u}r Diskrete Mathematik und Geometrie, Technische Universit\"{a}t Wien, Vienna, Austria\\\href{mailto: Jin Li<lijin2955@gmail.com>}{lijin2955@gmail.com}}
\date{}

\begin{document}

\maketitle

\begin{abstract}
Classifications of $\sln$ covariant function-valued valuations are established with some assumptions of continuity.
New valuations, for example, weighted moment functions, are introduced and our classifications give unified characterizations of the Laplace transform on convex bodies, $L_p$ moment bodies, $L_p$ difference bodies, and polar $L_p$ moment bodies ($L_{-p}$ intersection bodies).
Using the new classifications, we also establish some Euler-type relations.
\end{abstract}

\section{Introductions}
Let $\MK^n$ be the space of convex bodies (i.e., compact convex sets) in $n$-dimensional Euclidean space $\R^n$.
A valuation is a map $Z$ from a subspace of $\MK^n$ to an Abelian semigroup $\langle \MA, +\rangle$ such that
\begin{align}\label{val100}
Z K + Z L = Z (K \cup L) + Z(K \cap L),
\end{align}
whenever $K,L, K \cup L, K \cap L$ are contained in this subspace.
Started with the celebrated Hadwiger characterization theorem of intrinsic volumes, the theory of valuations has attracted great interest, especially in the last decades; see \cite{AW2015AFI,SW2015even,Ale99,Ale01,ABS2011harm,BF2011herm,CLM2017Min,HP14a,HP14b,HS2014loc,Kla95,Lud05,Lud06,LR10,SW2015mink,Sch2010,Joch2015comb}.
A valuation $Z$ taking values in a function space is a function-valued valuation if the addition in \eqref{val100} is the ordinary addition of functions.

This paper aims to give classifications of $\sln$ covariant function-valued valuations, only with some assumptions of continuity.
Meanwhile, we find new valuations, for example, weighted moment functions.
Our classifications establish unified characterizations of the Laplace transform on convex bodies \cite{LM2017Lap}; $L_p$ moment bodies and $L_p$ difference bodies for $p \ge 1$ ($L_p$ Minkowski valuations) \cite{Lud05,Hab12b,Wan11,Par14b,LL2016LpMV}; and polar $L_p$ moment bodies for $p>-1$ \cite{Hab09}.
The polar $L_{p}$ moment bodies are also called $L_{-p}$ intersection bodies which are characterized in \cite{HL06} for $0<-p<1$.
Polar $L_0$ moment bodies \cite{MR2244603} are not characterized in \cite{Hab09}, but they are also included in our classifications.
Above body-valued valuations are embedded in the space of function-valued valuations by their representations: support functions (for convex bodies) and gauge functions (for star bodies).

Let $\FF{n}$ be the space of all functions $f: \R^n \to \R$.
The special linear group is denoted by $\sln$.
A map $Z: \MK^n \to \FF{n}$ is $\sln$ covariant if
\begin{align*}
Z(\phi K)(x)=ZK(\phi^t x)
\end{align*}
for any $\phi \in \sln$ and $x \in \R^n$, where $\phi^t$ is the transpose of $\phi$.
If $Z: \MK^n \to \FF{n}$ is an $\sln$ covariant valuation, then the map $K \mapsto ZK(o)$ is an $\sln$ invariant real-valued valuation.
Classifications of $\sln$ invariant real-valued valuations were known before; see, e.g. \cite{LR2017sl} (or Corollary \ref{co:real}).
Hence we only deal with $\sln$ covariant valuations with values in $\FFo{n}$: the space of all functions $f: \ro{n} \to \R$.

Let $\mpo{n}$ be the space of polytopes containing the origin.
A signed Radon measure is called continuous if the measure of every singleton is $0$.
Equivalently, the (Lebesgue) decomposition of this measure only gives an absolutely continuous part and a singularly continuous part.
We denote by $\cms$ the space of signed and continuous Radon measures on $\R$.
Let $Z$ map a subspace of $\MK^n$ to $\FFo{n}$.
We call $Z$ continuous if $ZK_i (x) \to ZK(x)$ for every $x\in \ro{n}$ whenever $K_i,K$ are contained in this subspace and $K_i \to K$ with respect to the Hausdorff metric.

\begin{mthm}\label{mthm:sln3}
Let $n \ge 3$. A map $Z: \mpo{n} \to \FFo{n}$ is a continuous and $\sln$ covariant valuation
if and only if
there are $\zeta \in \CF{}$ and $\mu \in \cms$ such that
\begin{align*}
ZP(x)&= \zeta(h_P(x)) +\zeta(-h_{-P}(x)) + \frac{1}{|x|} \int_{\R} V_{n-1}(P \cap H_{x,t}) d\mu(t)
\end{align*}
for every $P \in \mpo{n}$ and $x \in \ro{n}$.
\end{mthm}

Here $h_P$ is the support function of $P$, and $V_{n-1}(P \cap H_{x,t})$ is the section volume (that is, the $(n-1)$-dimensional volume) of $P \cap H_{x,t}$, where $H_{x,t}=\{y\in\R^n: x\cdot y=t\}$.
The function $h_P^p+h_{-P}^p$ is the $p$-th power of the support function of $L_p$ difference body.
If $\mu$ has a locally integrable density $\zeta \in \FF{}$, by Fubini's theorem,
\begin{align}\label{eq:fubini}
\frac{1}{|x|}\int_{\R} V_{n-1}(K \cap H_{x,t}) \zeta(t) dt
=\int_{K} \zeta(x \cdot y) dy=:M_\zeta K(x)
\end{align}
for any $x \in \ro{n}$ and $K \in \MK^n$.
If $\zeta(t)=e^{-t}$ for $t \in \R$, then $M_\zeta K$ is the Laplace transform of $K$.
If $\zeta(t)=|t|^p$, then $M_\zeta K ^{1/p}$ is the support function of the $L_p$ moment body for $p \ge 1$ and the gauge function of the polar $L_p$ moment body ($L_{-p}$ intersection body) for $p>-1$, $p \neq 0$.
If $\zeta(t)=\log |t|$ for $t \neq 0$ and $0$ for $t=0$, then $\log M_\zeta K$ is the gauge function of the polar $L_0$ moment body.
We call $M_\zeta K$ the weighted moment function of $K$.
See \S \ref{sec:notations} for details.

Removing the continuity of $Z$ in Theorem \ref{mthm:sln3}, we obtain additional valuations even for valuations with values in $\CFr{n}$: the space of continuous functions $f$ on $\ro{n}$ such that $\lim_{r \to 0} rf(rx)=0$ for any $x \in \ro{n}$.
For a function $\zeta$ on $\R$, let $\zeta^R$ denote its reflection, that is, $\zeta^R (r)=\zeta(-r)$ for any $r\in\R^n$.
Let $\MP^n$ denote the space of all polytopes in $\R^n$.
We say a valuation $Z:\MP^n \to \CFr{n}$ is regular if, for any $P \in \MP^n$ and $x \in \ro{n}$, the function $s \mapsto Z(sP)(s^{-1}x)$ is regular: bounded or measurable on any compact interval $I \subset (0,\infty)$.
All assumptions of the following theorem are fulfilled for continuous and $\sln$ covariant $\FFo{n}$-valued valuations. 
Also, $f \in \CF{n}$ implies $\lim_{r \to 0} rf(ru)=0$ for any $u \in \Sn$.

\begin{mthm}\label{mthm:sln}
Let $n \ge 3$. A map $Z: \mpo{n} \to \CFr{n}$ is a regular and $\sln$ covariant valuation
if and only if
there are $\zeta_1,\zeta_2 \in \CF{}$ and $\mu \in \cms$ such that
\begin{align*}
ZP(x)&= \zeta_1(h_P(x)) +\zeta_1^R(h_{-P}(x)) + \eutr{-}{\zeta_2}(P)(x) + \eutr{-}{\zeta_2^R}(-P)(x) +\frac{1}{|x|} \int_{\R} V_{n-1}(P \cap H_{x,t}) d\mu(t)
\end{align*}
for every $P \in \mpo{n}$ and $x \in \ro{n}$.
\end{mthm}

For $P \in \mpo{n}$, the new valuation $\eutr{-}{\zeta}P$ is defined by
$$\eutr{-}{\zeta} P(x):= \sum_{F \in \face{-}{} (P)} (-1)^{\dim F} \zeta(h_{F}(x)),$$
where $\dim F$ is the dimension of $F$ and $\face{-}{}(P)$ is the class of faces (including $P$ itself) of $P$ which contain the origin.
It is related to Euler-type relations of support functions (Corollary \ref{co:eusupp}).
For more information on $\eutr{-}{\zeta}$ and related functionals, see \S \ref{sec:notations}.

For certain discontinuous $\zeta$, the valuation $P \mapsto \zeta(h_P)$ is still important in convex geometry.
One example is the characteristic function of the polar body.
Hence, we also characterize the valuations $\zeta(h_P)$ and $\eutr{-}{\zeta}(P)$ for all $\zeta \in \FF{}$; see \S \ref{sec:co0} for details.

Dual theorems, that is, classifications of $\sln$ contravariant $\CF{n}$-valued valuations, are established in \cite{Li2018AFV}. $\sln$ contravariant $\CF{n}$-valued valuations are related to Orlicz mixed volumes introduced in \cite{GHW14,XJL14}.
Our classification of $\sln$ covariant function-valued valuations might also contribute to the recently developing Orlicz Brunn-Minkowski theory and the $L_p$ Brunn-Minkowski theory for $p<1$ \cite{LYZ10a,LYZ10b,BLYZ12,BLYZ13,LR10,GHW14,XJL14}.

Recently, many interesting theorems for valuations on functions were established by extending the corresponding theorems for valuations on convex bodies; e.g. \cite{MR4070303,Lud12,Ale2017MAo,CLM2017Min,LM2017Lap,CLM2019homogeneous,MR3897436}.
Our results will also be used to give a characterization of the Legendre transform and the Laplace transform on convex functions; see \cite{Li2020LLV}.

The paper is organized as follows. In \S \ref{sec:co0}, we give a classification of valuations which are $\GL{+}{n}$ (or $\GL{}{n}$) covariance of weight $0$ so that the assumptions of continuity  or regularity in Theorem \ref{mthm:sln3} and Theorem \ref{mthm:sln} can be removed.
Notation and background information are collected in \S \ref{sec:notations}. We also give additional examples of valuations that do not take values at continuous functions.
In \S \ref{sec:val}, we prove the ``necessarity" of our theorems, that is, the ``if" part, that those operators are valuations which satisfy the desired properties, e.g., $\sln$ covariance, continuity.
The sections \ref{sec:fe}-\ref{sec:mainpr} are contributed to the proofs of the ``sufficiency", the ``only if" part.
Using our classifications, we establish new Euler-type relations and give new proofs for Euler-type relations introduced by Shephard \cite{MR232280}
and for local Euler-type relations introduced by Kabluchko, Last, and Zaporozhets \cite{MR3679943}; see \S \ref{sec:euler}.
Classifications of valuations on $\MP^n$ and $\MK^n$ are given in \S \ref{sec:genpoly}.
We pick up some explicit characterizations of some operators as applications of our theorems in the last section.
Some of them are known before and some of them are new.

\section{$\GL{+}{n}$ covariant valuations of weight $0$}\label{sec:co0}
The general linear group is denoted by $\GL{}{n}$ and the group of non-degenerate linear transforms with positive determinants by $\GL{+}{n}$.
A map $Z:\MK^n \rightarrow \FF{n}$ is $\GL{+}{n}$ (or $\GL{}{n}$) covariant of weight $0$ if
\begin{align*}
Z(\phi K)(x)= ZK ( \phi^{t}x)
\end{align*}
for any $\phi \in \GL{+}{n}$ (or $\GL{}{n}$, respectively) and $x \in \R^n$.
If $\zeta \in \FF{}$ is not continuous, the functions $\eutr{-}{\zeta}P$ and $\zeta(h_P)$ are not continuous in general.
We can characterize them by assuming that they are $\GL{+}{n}$ covariant of weight $0$.

\begin{mthm}\label{thm:co0}
Let $n \ge 3$.  A map $Z: \mpo{n} \to \FFo{n}$ is a $\GL{+}{n}$ covariant valuation of weight $0$ if and only if
there are $\zeta_1,\zeta_2 \in \FF{}$
such that
\begin{align}\label{val:co0}
ZP(x)=&\zeta_1(h_P(x)) +\zeta_1^R(h_{-P}(x)) + \eutr{-}{\zeta_2}(P)(x) + \eutr{-}{\zeta_2^R}(-P)(x)
\end{align}
for every $P \in \mpo{n}$ and $x \in \R^n$.

A map $Z: \mpo{2} \to \FFo{2}$ is a $\GL{}{2}$ covariant valuation of weight $0$ if and only if \eqref{val:co0} holds for $n=2$.
\end{mthm}

Let $\kik{K} (x)$ denote the characteristic function of a set $K \subset \R^n$, that is, if $x \in K$, then $\kik{K} (x)=1$, otherwise $\kik{K} (x)=0$.
Notice that the characteristic functions of polar bodies are included in the functions classified in Theorem \ref{thm:co0} since $\kik{P^\ast}(x)=\zeta(h_P(x))=\zeta(\|x\|_{P^\ast})$ for $\zeta(t)=\kik{[0,1]}(t)$.
Here $P^\ast =\{y \in \R^n : x \cdot y \leq 1\}$ and $\|x\|_{P^\ast}=\min\{\lambda >0: x \in \lambda K\}$ is the gauge function.
Also, the characteristic function of the normal cone of $P$ at the origin, $N(P,o)$, is included in the functions classified in Theorem \ref{thm:co0} by $\kik{N(P,o)}(x)=\zeta(h_P(x))$ for $\zeta(t)=\kik{0}(t)$.

\section{Preliminaries and Notation}\label{sec:notations}
We refer to Schneider \cite{Schb2} as a general reference for convex geometry and valuations theory.

Let $x \cdot y$ denote the inner product of $x,y \in \R^n$.
The Euclidean norm of $x \in \R^n$ is denoted by $|x|$.
The support function of a convex body $K$ is
\begin{align*}
h_K(x)=\max_{y \in K} \{x\cdot y\},~~x\in\R^n.
\end{align*}
The Hausdorff distance of $K,L \in \MK^n$ is $\delta(K,L)=\max_{u \in \Sn} |h_K(u) -h_L(u)|$.
The gauge function $\|x\|_{K^\ast}=h_K(x)$ if $K \in \MK^n$ contains the origin.

For $p \ge 1$, the $L_p$ moment body $M_p K$ of a convex body $K$ containing the origin, is defined by
\begin{align}\label{def:pmom}
h_{M_p K}(x)^p=\int_K |x \cdot y|^p dy, ~~x\in \R^n.
\end{align}
The $L_p$ moment body is a fundamental operator in convex geometry.
For the related $L_p$ affine isoperimetric inequality, see \cite{LYZ00a,LYZ00b,HS09a}.
It is called an $L_p$ Minkowski valuation since it is a valuation with respect to the $L_p$ Minkowski addition $h_{K+_p L}^p:=h_K^p+h_L^p$.
The volume-normalized $L_p$ moment bodies are call $L_p$ centroid bodies.

For $-1<p<1$, the function $(\int_K |x \cdot y|^p dy)^{1/p}$ is not a support function anymore.
But the polar $L_p$ moment body $M_p^\ast K=(M_p K)^\ast$, a star-shaped body, is still interesting.
The gauge function of $M_p^\ast K$ is defined by
\begin{align}\label{def:polarmom}
\|x\|_{M_p^\ast K}^p=\int_K |x \cdot y|^p dy, ~~x\in \R^n.
\end{align}
for $p>-1$, $p \neq 0$, and
\begin{align*}
V(K)\log \|x\|_{M_0^\ast K}=\int_K \log |x \cdot y| dy, ~~x\in \R^n.
\end{align*}
for $p=0$.
Polar $L_{-p}$ moment bodies are also called $L_p$ intersection bodies since
$$\lim_{p \uparrow 1}\Gamma(1-p) \|x\|_{M_{-p}^\ast K}^{-p} = \|x\|_{IK}^{-1},$$
where $\Gamma$ denotes the Gamma function; see \cite{Hab2008}.
Here $IK$ is the intersection body of $K$ which is defined by
\begin{align*}
\|x\|_{IK}^{-1}=\frac{1}{|x|}V_{n-1}(K \cap H_{x,0}).
\end{align*}
$L_p$ Busemann-Petty problems are studied in \cite{MR2244603,MR1658156}.

The Laplace transform $\lt K$ of a convex body $K$ is
\begin{align}\label{def:lap}
\lt K(x)=\int_{K} e^{-x \cdot y} dy,  ~~x\in \R^n.
\end{align}
Making use of the logarithmic version of this transform,
Klartag \cite{Kla06} found so far the best estimate on Bourgain's slicing problem, which is one of the main open problems in asymptotic geometry.
Also, the characterization of the Laplace transform on convex bodies is fundamental in the characterization of the classic Laplace transform on Lebesgue functions; see \cite{LM2017Lap}.

Comparing \eqref{def:pmom} with \eqref{def:polarmom} and \eqref{def:lap}, we are led to define the weighted moment function
\begin{align*}
M_\zeta K(x):=\int_K \zeta(x \cdot y) dy, ~~x\in \R^n, K \in \MK^n
\end{align*}
for a locally integrable function $\zeta: \R \to \R$.
Inspired by \eqref{eq:fubini}, we can further define the more general valuation:
\begin{align*}
M_\mu(K)(x)= \frac{1}{|x|} \int_{\R} V_{n-1}(K \cap H_{x,t}) d\mu(t), ~~x \in \ro{n}
\end{align*}
where $\mu$ is a signed Radon measure.
If $\mu$ is the Dirac-delta measure concentrated at $0$, then $M_\mu(K)=\|x\|_{IK}^{-1}$, which is characterized by Ludwig \cite{Lud06}.

We only characterize $M_{\mu}K$ for continuous measures since it is not continuous if $\mu$ is a discrete measure.

Another $\sln$ covariant valuation not included in our theorems is
$$K \mapsto \zeta(h_K(x)) V\left(K,x\right),$$
where $V\left(K,x\right)$ is the volume of the cone $\left\{\lambda y \in \R^n:  \lambda \in [0,1],~ y\cdot x=h_K\left(x\right)\right\}$ and $\zeta \in \FF{}$.
For some homogeneous $\zeta$, this valuation is characterized by Haberl and Parapatits \cite{HP14a} as an $\sln$ covariant valuation valued in a special function space,
and is used to characterize $L_p$ surface area measures.

Let $P \in \MP^n$.
We say a set $F$ is a face of $P$ if either $F=P$ or $F=P \cap H_{u,h_P(u)}$ for some $u \in \Sn$. The normal cone $N(P,F)=\{u \in \ro{n} : F \subseteq P \cap H_{u,h_P(u)}\} \cup \{o\}$.
We denote by $\face{}{}(P)$ the class of faces of $P$, and by $\face{-}{} (P)$ ($\face{+}{} (P)$) the class of faces $F$ of $P$ such that $h_P(u) \leq 0$ ($h_P(u) \ge 0$, respectively) for all $u \in N(P,F)$.
Let $\dim K$ be the dimension of $K \in \MK^n$, i.e., the dimension of the affine hull of $K$.
For any $\zeta \in \FF{}$, we define the maps $\eutr{\pm}{\zeta},\eutr{}{\zeta} : \MP^n \to \FF{n}$ by
\begin{align}
&\eutr{\pm}{\zeta} P(x):= \sum_{F \in \face{\pm}{} (P)} (-1)^{\dim F} \zeta(h_{F}(x)), \label{defval3} \\
&\eutr{}{\zeta} P(x):= \sum_{F \in \face{}{} (P)} (-1)^{\dim F} \zeta(h_{F}(x)) \label{defval4}
\end{align}
for any $P \in \MP^n$ and $x \in \R^n$.
If $o \in P$, then $\face{+}{}(P)=\face{}{}(P)$ and $\face{-}{} (P)$ is the class of faces of $P$ which contain the origin.

For $0 \leq j \leq \dim P$, let $\face{}{j} (P)=\{F \in \face{}{}(P): \dim F=j\}$ and $\face{\pm}{j} (P)=\{F \in \face{\pm}{}(P): \dim F=j\}$
Also, let $|\face{}{j} (P)|$ and $|\face{\pm}{j} (P)|$ be their cardinalities.

The convex hull of a set $A$ is denoted by $[A]$ and the convex hull of a set $A$ and a point $x \in \R^n$ will be briefly written as $[A,x]$ instead of $[A,\{x\}]$.

\section{Properties of the classified valuations}\label{sec:val}
We call a valuation simple if it vanishes on lower-dimensional sets.
Let $\CSF{n}$ be the space of infinitely differentiable functions on $\R^n$.
Denote by $\ACF{n}$ the space of continuous functions $f:\ro{n} \to \R$ satisfying that, for any $u \in \Sn$, \\
1. the function $r \mapsto rf(ru)$ is absolutely continuous on $(0,s]$ for any $s>0$; \\
2. $\lim_{r \to 0}r f(ru)=0$ for any $u \in \Sn$.

\begin{lem}\label{lem:momf}
Let $\mu \in \cms$.
The map $Z$ defined on $\MK^n$ by
\begin{align}\label{eq:meas1}
ZK(x)=\frac{1}{|x|}\int_{\R} V_{n-1}(K \cap H_{x,t}) d\mu(t),~~K \in \MK^n,~~x\in \ro{n},
\end{align}
is a continuous, simple, $\sln$ covariant $\CFr{n}$-valued valuation.

Moreover, if $\mu$ has a density $\zeta \in \FF{}$, then
\begin{align}\label{eq:ac1}
ZK(x)=\int_{K} \zeta(x \cdot y) dy,
\end{align}
and $ZK \in \ACF{n}$ for every $K \in \MK^n$.
Also, if $\zeta \in \CSF{}$, \eqref{eq:ac1} can be defined also at $x=o$ and $ZK \in \CSF{n}$ for every $K \in \MK^n$.
\end{lem}
\begin{proof}
The simplicity and the valuation property of $Z$ are trivial.

Since the valuation defined by \eqref{eq:ac1} is clearly $\sln$ covariant, by \eqref{eq:fubini} and a standard approximation, the valuation defined by \eqref{eq:meas1} is also $\sln$ covariant.

Let $\{x_i\} \subset \ro{n}$ be a sequence converging to $x \in \ro{n}$ and $\{K_i\} \subset \MK^n$ be a sequence converging to $K \in \MK^n$.
By \cite[Theorem 1.8.8]{Schb2}, $K_i \cap H_{x_i,t} \to K \cap H_{x,t}$ if either $H_{x,t}$ passes through the relative interior of $K$ or $K \cap H_{x,t} = \emptyset$.
Hence $V_{n-1}(K_i \cap H_{x_i,t}) \to V_{n-1}(K \cap H_{x,t})$ for every $t \in \R \setminus \{h_K(x), -h_K(-x)\}$.
Notice that the support set of the function $t \mapsto V_{n-1}(K \cap H_{x,t})$ is compact.
By the dominated convergence theorem, \eqref{eq:meas1}, the boundedness of $t \mapsto V_{n-1}(K \cap H_{x,t})$ and
the continuity of $\mu$, we obtain $ZK_i(x_i) \to ZK(x)$.
This not only gives the continuity of $ZK$, but also the continuity of $Z$.

For $x \in \ro{n}$, let $r=|x|$ and $u=x/r$.
By \eqref{eq:meas1},
\begin{align*}
rZK(ru)=\int_{-rh_{-K}(u)}^{rh_{K}(u)} V_{n-1}(K \cap H_{ru,t}) d\mu(t).
\end{align*}
Since $\mu$ is continuous and $$V_{n-1}(K \cap H_{ru,t}) \le \max_{t \in \R} V_{n-1}(K \cap H_{u,t}),$$
we have $\lim_{r \to 0} rZK(ru)=0$.

If $\mu$ has a density $\zeta$, then $\zeta$ is locally integrable.
For any $\e>0$, since $|\zeta(t) V_{n-1}(K \cap H_{ru,t})| \le |\zeta(t)| \max_{t \in \R} V_{n-1}(K \cap H_{u,t})$ and $\zeta$ is locally integrable,
there is a $\delta>0$ such that
\begin{align*}
\left|\int_{E} \zeta(t) V_{n-1}(K \cap H_{ru,t})dt \right| < \e,
\end{align*}
whenever $|E|<\delta$ for compact $E \subset \R$.
Hence the function
\begin{align*}
r \mapsto rZK(ru)=\int_{0}^{rh_{K}(u)} \zeta(t)V_{n-1}(K \cap H_{ru,t}) d\mu(t) +\int_{-rh_{-K}(u)}^{0} \zeta(t) V_{n-1}(K \cap H_{ru,t}) d\mu(t)
\end{align*}
is absolutely continuous on $(0,s]$ for any $s>0$.

If further $\zeta \in \CSF{}$, then clearly $ZK \in \CSF{n}$ for any $K \in \MK^n$.
\end{proof}

Since
\begin{align*}
h_{K \cup L} =\max\{h_K,h_L\}, h_{K \cap L} =\min\{h_K,h_L\}
\end{align*}
for $K,L,K \cup L \in \MK^n$, the map $K \mapsto \zeta(h_K)$ is a valuation.
The following two lemmas are trivial, so we omit the proofs.
\begin{lem}\label{lem:suppf}
Let $\zeta \in \FF{n}$. The map $Z:\MK^n \to \FF{n}$, defined by $ZK=\zeta(h_K)$, is a $\GL{}{n}$ covariant valuation of weight $0$.
Moreover, if $\zeta \in \CF{}$, then $ZK \in \CF{n}$ and $Z$ is continuous.
\end{lem}

\begin{lem}\label{lem:ref}
If $Z:\MK^n \to \FF{n}$ is a continuous, $\GL{}{n}$ covariant valuation of weight $0$, so are $Z', Z''$ defined by $Z'K=Z(-K)$ and $Z''K=Z[K,o]$.
\end{lem}

To check that $\eutr{\pm}{\zeta}$, $\eutr{}{\zeta}$ are valuations on $\MP^n$, we only need to check that \eqref{val100} holds for some special $K,L\in \mpo{n}$.

\begin{lem}[\cite{Schb2}, Theorem 6.2.3]\label{lem:wval}
If a map $Z:\MP^n \to \R$ satisfies that
\begin{align}\label{eq:wval1}
Z(P) + Z(P \cap H) = Z(P \cap H^+) + Z(P \cap H^-)
\end{align}
for any $P \in \MP^n$ and hyperplane $H$, then $Z$ is a valuation.
Here $H^+,H^-$ denote the half-spaces bounded by $H$.
\end{lem}

\begin{lem}\label{lem:val1}
The operators $\eutr{\pm}{\zeta}:\MP^n \to \FF{n}$ and $\eutr{}{\zeta}:\MP^n \to \FF{n}$ are $\GL{}{n}$ covariant valuations of weight $0$.
Moreover, if $\zeta \in \CF{}$, then $\eutr{\pm}{\zeta}(P) \in \CF{n}$ for any $P \in \MP^n$.
\end{lem}
\begin{proof}
By the definitions \eqref{defval3} and \eqref{defval4} together with Lemma \ref{lem:suppf}, $\eutr{\pm}{\zeta}$ and $\eutr{}{\zeta}$ are $\GL{}{n}$ covariant of weight $0$ and the last statement holds.

In the following, we only show that $\eutr{-}{\zeta}$ is a valuation. The argument for $\eutr{+}{\zeta}$ and $\eutr{}{\zeta}$ are similar.
By Lemma \ref{lem:wval}, we only need to check that \eqref{eq:wval1} holds for any $P \in \MP^n$ and hyperplane $H$.
We can assume that $\dim (P \cap H) < \dim P=\dim (P \cap H^+)=\dim (P \cap H^-)$.
We can also assume that $\dim P =n$ since otherwise, we can consider the affine hull of $P$ for the following argument.
Now we have $\dim (P \cap H) =n-1$.

Let $0\leq j \leq n$.
We can separate $\face{}{j} (P)$ as follows
\begin{align}\label{eq:face1}
\face{}{j} (P) = \mathcal{A}_{j,H^+} \cup \mathcal{A}_{j,H^-} \cup \mathcal{B}_j \cup \mathcal{C}_j,
\end{align}
where
\begin{align*}
&\mathcal{A}_{j,H^{\pm}} =\{F \in \face{}{j} (P): \relint F \subset \relint H^\pm\}, \\
&\mathcal{B}_j=\{F \in \face{}{j} (P): F \subset H\}, \\
&\mathcal{C}_j=\{F \in \face{}{j} (P): F \cap \relint H^+ \neq \emptyset,  F \cap \relint H^- \neq \emptyset\}.
\end{align*}
Set $\mathcal{C}_{n+1}= \emptyset$ and notice that $\mathcal{C}_n=\{P\}$.
Since $\partial (P\cap H) = \partial P \cap H$ and $\partial (P\cap H^\pm)=(\partial P\cap H^\pm) \cup (P \cap H)$,
it is easy to see
\begin{align}\label{eq:face2}
\face{}{j} (P\cap H) = \mathcal{B}_j \cup \{F \cap H: F \in \mathcal{C}_{j+1}\},
\end{align}
and
\begin{align}
\face{}{j} (P\cap H^+) = \mathcal{A}_{j,H^+} \cup \mathcal{B}_j \cup \{F \cap H^+: F \in \mathcal{C}_j\} \cup \{F \cap H: F \in \mathcal{C}_{j+1}\}, \label{eq:face3} \\
\face{}{j} (P\cap H^-) = \mathcal{A}_{j,H^-} \cup \mathcal{B}_j \cup \{F \cap H^-: F \in \mathcal{C}_j\} \cup \{F \cap H: F \in \mathcal{C}_{j+1}\}. \label{eq:face4}
\end{align}

Let $\mathcal{A}_{j,H^{\pm}}^-=\mathcal{A}_{j,H^{\pm}} \cap \face{-}{j}(P)$, $\mathcal{B}_j^-=\mathcal{B}_j \cap \face{-}{j}(P)$ and $\mathcal{C}_j^-=\mathcal{C}_j \cap \face{-}{j}(P)$.

For the case $o \in H$, \eqref{eq:face1}-\eqref{eq:face4} also hold if we replace $\face{}{},\mathcal{A}_{j,H^{\pm}},\mathcal{B}_j,\mathcal{C}_j$ by $\face{-}{},\mathcal{A}_{j,H^{\pm}}^-,\mathcal{B}_j^-,\mathcal{C}_j^-$, respectively.
Now combined with the definition \eqref{defval3},
we have
\begin{align*}
&\eutr{-}{\zeta} P
    = \sum_{0 \leq j \leq n} (-1)^j\left(\sum_{F \in \mathcal{A}_{j,H^+}^-} \zeta (h_{F}) + \sum_{F \in \mathcal{A}_{j,H^-}^-} \zeta (h_{F}) + \sum_{F \in \mathcal{B}_j^-} \zeta (h_{F}) + \sum_{F \in \mathcal{C}_j^-} \zeta (h_{F})\right), \\
&\eutr{-}{\zeta} (P\cap H)
    = \sum_{0 \leq j \leq n-1} (-1)^j \left(\sum_{F \in \mathcal{B}_j^-} \zeta (h_{F}) + \sum_{F \in \mathcal{C}_{j+1}^-} \zeta(h_{F \cap H})\right) \\
&\eutr{-}{\zeta} (P\cap H^+)
    = \sum_{0 \leq j \leq n} (-1)^j \left(\sum_{F \in \mathcal{A}_{j,H^+}^-}  \zeta (h_{F}) + \sum_{F \in \mathcal{B}_j^-} \zeta (h_{F}) + \sum_{F \in \mathcal{C}_j^-} \zeta(h_{F \cap H^+}) + \sum_{F \in \mathcal{C}_{j+1}^-} \zeta(h_{F \cap H})\right),  \\
&\eutr{-}{\zeta} (P\cap H^-)
    = \sum_{0 \leq j \leq n} (-1)^j \left(\sum_{F \in \mathcal{A}_{j,H^-}^-} \zeta (h_{F}) + \sum_{F \in \mathcal{B}_j^-} \zeta (h_{F}) + \sum_{F \in \mathcal{C}_j^-} \zeta(h_{F \cap H^-}) + \sum_{F \in \mathcal{C}_{j+1}^-} \zeta(h_{F \cap H})\right).
\end{align*}
Notice that $\mathcal{B}_n^- =\emptyset$.
To check \eqref{eq:wval1} holds for $Z=\eutr{-}{\zeta}$, after eliminating the same items, we only need to show
\begin{align}\label{eq:eu3}
&\sum_{0 \leq j \leq n} (-1)^j \sum_{F \in \mathcal{C}_j^-} \zeta (h_{F}) + \sum_{0 \leq j \leq n-1} (-1)^j \sum_{F \in \mathcal{C}_{j+1}^-} \zeta (h_{F \cap H}) \notag \\
&=\sum_{0 \leq j \leq n} (-1)^j \sum_{F \in \mathcal{C}_j^-} \zeta (h_{F \cap H^+}) + \sum_{0 \leq j \leq n} (-1)^j \sum_{F \in \mathcal{C}_j^-} \zeta (h_{F \cap H^-}) + 2 \sum_{0 \leq j \leq n-1} (-1)^j \sum_{F \in \mathcal{C}_{j+1}^-} \zeta (h_{F \cap H}).
\end{align}
The equation \eqref{eq:eu3} is actually true since
\begin{align*}
\zeta(h_F) + \zeta (h_{F \cap H}) &= \zeta (h_{F \cap H^+}) + \zeta (h_{F \cap H^-}), 
\end{align*}
and $\mathcal{C}_{0}^-= \emptyset$ (by its definition).

For the case $o \notin H$,
\begin{align}\label{eq:face5}
\face{-}{j} (P) = \mathcal{A}_{j,H^+}^- \cup \mathcal{A}_{j,H^-}^- \cup \mathcal{B}_j^- \cup \mathcal{C}_j^-,
\end{align}
and
\begin{align}\label{eq:face6}
\face{-}{j} (P\cap H) = \emptyset.
\end{align}
Set $\mathcal{B}_{j,H^{\pm}}^-=\mathcal{B}_j^- \cap \face{-}{j}(P\cap H^\pm)$ and $\mathcal{D}_j^-=\{F \cap H: F \in \mathcal{C}_{j+1}^-\}$, $\mathcal{D}_{j,H^{\pm}}^-= \mathcal{D}_j^- \cap \face{-}{j}(P\cap H^\pm)$.
We have
\begin{align}
\face{-}{j} (P\cap H^+) = \mathcal{A}_{j,H^+}^- \cup \mathcal{B}_{j,H^+}^- \cup \{F \cap H^+: F \in \mathcal{C}_j^-\} \cup \mathcal{D}_{j,H^+}^-, \label{eq:face7} \\
\face{-}{j} (P\cap H^-) = \mathcal{A}_{j,H^-}^- \cup \mathcal{B}_{j,H^-}^- \cup \{F \cap H^-: F \in \mathcal{C}_j^-\} \cup \mathcal{D}_{j,H^-}^-, \label{eq:face8}
\end{align}
and
\begin{align}
&\mathcal{B}_{j,H^+}^- \cup \mathcal{B}_{j,H^-}^-=\mathcal{B}_j^-,
~~~\mathcal{B}_{j,H^+}^- \cap \mathcal{B}_{j,H^-}^-=\emptyset, \label{eq:face9}\\
&\mathcal{D}_{j,H^+}^- \cup \mathcal{D}_{j,H^-}^-=\mathcal{D}_j^-,
~~\mathcal{D}_{j,H^+}^- \cap \mathcal{D}_{j,H^-}^-=\emptyset.\label{eq:face10}
\end{align}
Combining \eqref{eq:face5}-\eqref{eq:face10} with the definition of $\eutr{-}{\zeta}$, we only need to show that
\begin{align*}
&\sum_{0 \leq j \leq n} (-1)^j \sum_{F \in \mathcal{C}_j^-} \zeta (h_{F})\\
&=\sum_{0 \leq j \leq n} (-1)^j \sum_{F \in \mathcal{C}_j^-} \zeta (h_{F \cap H^+}) + \sum_{0 \leq j \leq n} (-1)^j \sum_{F \in \mathcal{C}_j^-} \zeta (h_{F \cap H^-}) + \sum_{0 \leq j \leq n-1} (-1)^j \sum_{F \in \mathcal{C}_{j+1}^-} \zeta (h_{F \cap H}),
\end{align*}
which is the same as \eqref{eq:eu3}.
\end{proof}

Let $\CFo{n}$ be the space of continuous functions on $\ro{n}$.
We say a function $f: (0,\infty) \to \R$ is regular if $f$ is bounded or measurable on any compact interval.
\begin{lem}\label{lem:contin}
If $Z: \mpo{n} \to \FFo{n}$ is continuous and $\sln$ covariant, then $ZP \in \CFo{n}$ for every $P \in \mpo{n}$ and the function $s \mapsto Z(sT^n)(s^{-1}x)$ is regular on $(0,\infty)$ for any $x \in \ro{n}$.
Moreover, if further $Z$ is simple and $Z(sT^n)(s^{-1}x)=s^nZ(T^n)(x)$ for any $s>0$ and $x \in \ro{n}$, then $\lim_{s \to 0} sZT^n(sx)=0$ for any $x \in \ro{n}$.

If $Z: \MP^n \to \FFo{n}$ is continuous and $\sln$ covariant, then $ZP \in \CFo{n}$ for every $P \in \MP^n$ and the functions $s \mapsto Z(sT^n)(s^{-1}x)$ and $s \mapsto Z(s[e_1,\dots,e_n])(s^{-1}x)$ are regular on $(0,\infty)$ for any $x \in \ro{n}$.
Moreover, if further $Z$ vanishes on all $P\in \mpo{n}$ with $\dim P<n$, 
and $Z(sT^n)(s^{-1}x)=s^nZ(T^n)(x)$ and $Z(s[e_1,\dots,e_n])(s^{-1}x)=s^nZ([e_1,\dots,e_n])(x)$ for any $x \in \ro{n}$, then $\lim_{s \to 0} sZT^n(sx)=0$ and $\lim_{s \to 0} sZ[e_1,\dots,e_n](sx)=0$ for any $x \in \ro{n}$.
\end{lem}
\begin{proof}
We only prove the first part of the Lemma. The second part is similar.

Let $x_i,x \in \ro{n}$ such that $x_i \to x$.
We need to show that $ZP(x_i) \to ZP(x)$.
There are $\phi_i \in \sln$ such that $x_i=\phi_i^t x$ and $(\phi_i P) \to P$ with respect to the Hausdorff metric.
Hence the continuity and the $\sln$ covariance of $Z$ give $ZP(x_i)=Z(\phi_i P)(x) \to ZP(x)$.

Let $\phi_s \in \sln$ such that $\phi_s x = s^{-1} x$ and $\phi_s y=s^{1/(n-1)}y$ for any $y \in x^\bot:=\{y \in \R^n:x \cdot y=0\}$.
Clearly $\phi_s^t=\phi_s$.
The $\sln$ covariance of $Z$ gives
$Z(sT^n)(s^{-1}x)=Z(\phi_s sT^n)(x)$.
Hence, the continuity of $Z$ implies that $s \mapsto Z(sT^n)(s^{-1}x)$ is continuous which is clearly regular.

Lastly, let $s>0$ and $\psi \in \GL{+}{n}$ such that $\phi x = s x$ and $\phi y =y$ for any $y \in x^\bot$.
By the $\sln$ covariance of $Z$ and the assumption that $Z(sT^n)(s^{-1}x)=s^nZ(T^n)(x)$ for any $x \in \ro{n}$, we have
$Z(\psi T^n)(x)=Z(s^{1/n}T^n)(s^{-1/n}\psi^t x)=s ZT^n(sx)$.
Also, $\psi T^n$ converges to the orthogonal projection of $T^n$ to $x^\bot$ as $r \to 0$.
Together with the simplicity of $Z$, we get the desired property.
\end{proof}

\section{Functional equations and uniqueness arguments}\label{sec:fe}
If a regular function $f:(0,\infty) \to \R$ satisfies the Cauchy functional equation
\begin{align*}
f(a+b) = f(a) + f(b),
\end{align*}
for any $a,b> 0$, then
\begin{align*}
  f(r) = f(1)r
\end{align*}
for any $r>0$.

The following dissections of standard simplices $T^d:=[o,e_1,\dots,e_d]$ are fundamental in our arguments.

For $0 < \lambda <1$, the hyperplane $H_\lambda$ and half spaces $H_\lambda^-,H_\lambda^+$ are defined by
\begin{align*}
H_\lambda := \{x \in\R^n : x \cdot ((1-\lambda) e_1- \lambda e_2) = 0\}, \\
H_\lambda^- := \{ x \in\R^n : x \cdot ((1-\lambda) e_1- \lambda e_2) \leq 0 \}, \\
H_\lambda^+ := \{ x \in\R^n : x \cdot ((1-\lambda) e_1- \lambda e_2) \geq 0 \}.
\end{align*}
Let $Z: \mpo{n} \to \FFo{n}$ be an $\sln$ covariant valuation.
We have
\begin{align}\label{val}
Z (sT^{d}) (x) + Z (sT^{d} \cap H_\lambda) (x) = Z (sT^{d} \cap H_\lambda ^-) (x) + Z (sT^{d} \cap H_\lambda ^+) (x), ~~x \in \ro{n}
\end{align}
for any $s > 0$ and $2 \leq d \leq n$.
Also, let $\widehat{T}^{d-1} := [o,e_1,e_3,\dots,e_d]$ and $\phi_\lambda,\psi_\lambda \in \gln$ such that
$$\phi_\lambda e_1 = \lambda e_1 + \left(1-\lambda\right) e_2,~\phi_\lambda e_i = e_i,~\text{for}~2 \leq i \leq n,$$
and
$$\psi_\lambda e_2 = \lambda e_1 + \left(1-\lambda\right) e_2,
\psi_\lambda e_i = e_i,~\text{for}~1 \leq i \leq n, i \neq 2.$$
Notice that
$sT^{d}\cap H_\lambda ^- = \phi_\lambda sT^{d}$, $sT^{d}\cap H_\lambda ^+ = \psi_\lambda sT^{d}$ and $sT^d \cap H_\lambda =\phi_\lambda s\hat{T}^{d-1}$.
Combined with the $\SLn$ covariance of $Z$, \eqref{val} implies
\begin{align}\label{30}
&Z(sT^{d}) (x) + Z(\lambda^{1/n}s\hat{T}^{d-1}) (\lambda^{-1/n} \phi_\lambda ^t x) \notag \\
&\qquad =Z(\lambda^{1/n} sT^{d}) \left(\lambda^{-1/n} \phi_\lambda ^t x \right) + Z((1-\lambda)^{1/n}sT^{d}) \left((1-\lambda)^{-1/n}\psi_\lambda ^t x\right),
\end{align}
where $x=r_1e_1 + \dots + r_n e_n \in \ro{n}$, $\phi_\lambda ^t x =(\lambda r_1 + (1-\lambda)r_2)e_1+ r_2e_2+r_3e_3+\dots+ r_ne_n$ and $\psi_\lambda ^t x = r_1e_1+ (\lambda r_1 + (1-\lambda)r_2)e_2+r_3e_3+\dots+ r_ne_n$.

Now we establish some properties of $Z(sT^{d}) (x)$ by assuming $Z(s\hat{T}^{d-1}) (x)=0$ in \eqref{30}.

\begin{lem}\label{lem:fe1}
Let $d \geq 2$ and let the function $f: (0,\infty) \times \ro{d}  \to \R$ satisfy the following properties. \\
($\romannumeral1$)
\begin{align}\label{30a}
f \big(s^{1/n};y \big)&=f\big((\lambda s)^{1/n}; \lambda^{-1/n}\phi_\lambda^t y \big) + f \big(((1-\lambda) s)^{1/n}; (1-\lambda)^{-1/n}\psi_\lambda ^t y \big)
\end{align}
for any $s>0$, $\lambda \in (0,1)$ and $y \in \ro{d}$. \\
($\romannumeral2$) For any $s>0$, $y \in \ro{d}$ and an arbitrary permutation $\pi$ (even permutation for $d \ge 3$) of coordinates of $y$,
\begin{align}\label{perm}
f(s;y)=f(s;\pi y).
\end{align}

We have following conclusions. \\
(a) If $f(s;re_1+\dots+re_m)=0$ for any $s>0$, $r \neq 0$ and $1 \leq m \leq d$, then
\begin{align}\label{eq:unq1}
f(s;x)=0
\end{align}
for every $s>0$ and $x \in \ro{d}$.
Moreover, if $f(s; \cdot)$ is continuous, we only need the assumption $f(s;re_1)=0$ to conclude \eqref{eq:unq1}. \\
(b) For $d \geq 3$, if the function $s \mapsto f(s^{1/n};s^{-1/n}r(e_1+\dots+e_m))$ is regular for any $r \neq 0$ and $1 \leq m \leq d$, then
\begin{align}\label{eq:hom}
f(s^{1/n};s^{-1/n}x)=sf(1;x)
\end{align}
for any $x \in \ro{d}$.

For $d=2$, if the function $s \mapsto f(s^{1/n};s^{-1/n}r(e_1+e_2))$ is regular for any $r \neq 0$, then
\begin{align*}
f(s^{1/n};s^{-1/n}r(e_1+e_2))=sf(1;r(e_1+e_2)).
\end{align*}
(c) For $d \geq 3$, if $f(s^{1/n};s^{-1/n}x)=f(1;x)$ for any $s>0$ and $x \in \ro{d}$, then \eqref{eq:unq1} also holds.
For $d=2$, if additionally, $f(s;re_1)=0$ for any $s>0$ and $r \neq 0$, then \eqref{eq:unq1} also holds.
\end{lem}
\begin{proof}
(a). We prove \eqref{eq:unq1} by induction on $k$, the number of non-zero coordinates of $x$.
By \eqref{perm}, we only need to show that $f(s;x)=0$ for $x=r_1e_1+\dots+r_ke_k$, where $r_1,\dots r_k$ are non-zero numbers.
For $k=1$, it is trivial. Assume that the statement holds for $k-1$.
Set $\dot{x}=r_3e_3+\dots+r_ke_k$.

For $r_1 > r_2 > 0$ or $0 > r_2 >r_1$, taking $\lambda = \frac{r_2}{r_1}$, $y=(1-\lambda)^{1/n}(r_1e_1+\dot{x})$ and replacing $s$ with $s/(1-\lambda)$ in (\ref{30a}), we get
\begin{align}\label{38}
&f\left(\left(\frac{r_1}{r_1-r_2}s\right)^{1/n}; \left(\frac{r_1-r_2}{r_1}\right)^{1/n}(r_1e_1 + \dot{x})\right)  \nonumber \\
 &= f\left(\left(\frac{r_2}{r_1-r_2}s\right)^{1/n};\left(\frac{r_1-r_2}{r_2}\right)^{1/n}(r_2 e_1 + \dot{x})\right)  + f\left(s^{1/n};r_1 e_1 + r_2 e_2 +\dot{x}\right).
\end{align}
For $r_2 > r_1 > 0$ or $0 > r_1 >r_2$,
taking $1-\lambda = \frac{r_1}{r_2}$, $y=\lambda^{1/n}(r_2e_2+\dot{x})$ and replacing $s$ with $s/\lambda$ in (\ref{30a}), we get
\begin{align}\label{40}
&f \left(\left(\frac{r_2}{r_2-r_1}s\right)^{1/n};\left(\frac{r_2-r_1}{r_2}\right)^{1/n}(r_2e_2 + \dot{x})\right) , \nonumber \\
&= f \left(s^{1/n};r_1 e_1 +r_2 e_2 + \dot{x}\right) + f\left(\left(\frac{r_1}{r_2-r_1}s\right)^{1/n};\left(\frac{r_2-r_1}{r_1}\right)^{1/n}(r_1 e_2 +\dot{x})\right).
\end{align}
For $r_1 >0 > r_2$ or $r_2 > 0 > r_1$, taking $0 < \lambda = \frac{r_2}{r_2 - r_1} <1$ and $y=x$ in (\ref{30a}), we get
\begin{align}\label{23}
&f\left(s^{1/n};r_1 e_1 + r_2 e_2 + \dot{x}\right) \nonumber \\
= &  f\left(\left(\frac{r_2}{r_2 - r_1}s\right)^{1/n};\left(\frac{r_2 - r_1}{r_2}\right)^{1/n}(r_2 e_2 + \dot{x})\right) + f\left(\left(\frac{-r_1}{r_2 - r_1}s\right)^{1/n};\left(\frac{r_2 - r_1}{-r_1}\right)^{1/n}(r_1 e_1 + \dot{x})\right).
\end{align}
The induction assumption together with (\ref{38}), (\ref{40}) and (\ref{23}) implies that $f(s;x)=0$ for $x=r_1e_1+\dots+r_ke_k$ if $r_1 \neq r_2$.

Now let $r_1=r_2=:r$.
If $f(s; \cdot)$ is continuous, we get $f(s;re_1+re_2+r_3e_3+ \cdots +r_ke_k)=0$ follows from the above conclusions.
Without the continuity of $f(s;\cdot)$,
if all $r_1,\dots,r_k$ are the same, the case is what we assumed.
Hence we can say $r_3 \neq r$.
Also, $f(s;re_1+re_2+r_3e_3+\dots+r_ke_k)=f(s;re_1+r_3e_2+re_3+\dots+r_ke_k)$ by \eqref{perm}.
Now using the above argument again, we get $f(s;re_1+re_2+r_3e_3+r_4e_4+\dots+r_ke_k)=0$.

(b) We also use induction on $k$.
Let first $d \ge 3$. For any $r\neq 0$, choosing $y=s^{-1/n}re_d$ in (\ref{30a}), we have
\begin{align}\label{30-3}
f \big(s^{1/n};s^{-1/n}re_d \big)
=f \big((\lambda s)^{1/n};(\lambda s)^{-1/n} r e_d \big) + f \big( ((1-\lambda)s)^{1/n};((1-\lambda)s)^{-1/n}r e_d \big)
\end{align}
for any $0 < \lambda <1$ and $s>0$.
Define
\begin{align*}
\xi_r(s) = f(s^{1/n};s^{-1/n}re_d)
\end{align*}
for $s>0$ and $r\in\R$.
For arbitrary $a,b > 0 $, $r \in \R$, set $s= a+b$, $\lambda =\frac{a}{a+b}$ in (\ref{30-3}).
We get that $\xi_r$ satisfies the Cauchy functional equation.
By the assumption, $\xi_r$ is regular. 
Hence
\begin{align*}
f(s^{1/n};s^{-1/n}re_d)=s f(1;re_d)
\end{align*}
for any $s>0$ and $r \neq 0$.
By \eqref{perm}, we finish the cases of $d \ge 3$ and $k=1$.

For $k \geq 2$, we prove the case $d=2$ and the cases $d \ge 3$ together.
Assume that the desired results hold for $k-1$.
If there are two different non-zero coordinates of $x$, the induction assumption together with \eqref{perm}, (\ref{38}), (\ref{40}) and (\ref{23}) implies the desired result.
If $r_1=r_2=\dots=r_k=:r$, taking $y= s^{-1/n}(re_1 +\dots +re_k)$ in (\ref{30a}), we get
\begin{align*}
&f\left(s^{1/n};s^{-1/n}(re_1 +\dots +re_k)\right)\nonumber \\
= &  f\left((\lambda s)^{1/n};(\lambda s)^{-1/n}(re_1 +\dots +re_k)\right) + f\left(((1-\lambda)s)^{1/n};((1-\lambda)s)^{-1/n}(re_1 +\dots +re_k)\right),
\end{align*}
which implies that the function $s \mapsto f\left(s^{1/n};s^{-1/n}(re_1 +\dots +re_k)\right)$ satisfies the Cauchy functional equation.
Since it is regular, $f\left(s^{1/n};s^{-1/n}(re_1 +\dots +re_k)\right)= sf\left(1;re_1 +\dots +re_k\right)$.

(c) For $d \ge 3$, by (a), we only need to show that $f\left(s^{1/n};s^{-1/n}(re_1 +\dots +re_m)\right)=0$ for any $1 \leq m \leq d$, which follows directly by the assumption in (c) and \eqref{eq:hom}.
The conclusion for $d=2$ is similar.
\end{proof}

The following uniqueness theorems are also fundamental. The proofs are similar to the corresponding proofs for $\sln$ contravariant valuations \cite{Li2018AFV}.
\begin{lem}\label{lemuq}
Let $Z$ and $Z'$ be $\SLn$ covariant function-valued valuations on $\mpo{n}$. If $Z (sT^d) = Z' (sT^d)$ for every $s>0$ and $0 \leq d \leq n$, then $Z P = Z' P$ for every $P \in \mpo{n}$.
\end{lem}

\begin{lem}\label{lemuq2}
Let $Z$ and $Z'$ be $\SLn$ covariant function-valued valuations on $\MP^n$. If $Z (sT^d) = Z' (sT^d)$ and $Z(s[e_1,\dots,e_d]) = Z'(s[e_1,\dots,e_d])$ for every $s>0$ and $0 \leq d \leq n$, then $Z P = Z' P$ for every $P \in \MP^n$.
\end{lem}

\section{Simple valuations}
In this section, we first prove the following theorems.

\begin{thm}\label{mthm:simpsln}
Let $n \ge 3$. A map $Z: \mpo{n} \to \CFr{n}$ is a regular, simple and $\sln$ covariant valuation
if and only if
there is a measure $\mu \in \cms$ such that
\begin{align*}
ZP(x)= \frac{1}{|x|} \int_{\R} V_{n-1}(P \cap H_{x,t}) d\mu(t)
\end{align*}
for every $P \in \mpo{n}$ and $x \in \ro{n}$.
\end{thm}

Remark: for the ``only if part", we only need the weak assumption that $Z: \mpo{n} \to \CFo{n}$ is a simple and $\sln$ covariant valuation such that $ZT^n \in \CFr{n}$ and the function $s \mapsto Z(sT^n)(s^{-1}x)$ is regular on $(0, \infty)$.

The characterization of $M_\zeta$ may be of special interest.

\begin{thm}\label{mthm:acsimpsln}
Let $n \ge 3$. A map $Z: \mpo{n} \to \ACF{n}$ is a regular, simple and $\sln$ covariant valuation
if and only if
there is a locally integrable function $\zeta$ on $\R$ such that
\begin{align*}
ZP(x)= \int_{P} \zeta(x \cdot y) dy
\end{align*}
for every $P \in \mpo{n}$ and $x \in \ro{n}$.
\end{thm}

For smooth $\zeta$, we obtain the following result.

\begin{thm}\label{mthm:smsimsln}
Let $n \ge 3$. A map $Z: \mpo{n} \to \CSF{n}$ is a regular, simple and $\sln$ covariant valuation
if and only if
there is a function $\zeta \in \CSF{}$ such that
\begin{align*}
ZP(x)&= \int_{P} \zeta(x \cdot y) dy
\end{align*}
for every $P \in \mpo{n}$ and $x \in \R^n$.
\end{thm}

The "if" part of all theorems follows directly from Lemma \ref{lem:momf}.

Let
\begin{align}\label{eq:fe2}
f(s;y)=Z(sT^n)(y),~~y\in \ro{n}.
\end{align}
Since $Z$ is a simple and $\sln$ covariant valuation, by \eqref{30}, Lemma \ref{lem:fe1} (a) and Lemma \ref{lemuq}, we only need the following lemma.

\begin{lem}\label{lem:simple1}
Let $n \geq 3$. If $Z : \mathcal{P}_o ^n \to \CFr{n}$ is a regular, simple and $\SLn$ covariant valuation,
then there is a measure $\mu \in \cms$ such that
\begin{align*}
Z(sT^n) (re_1)
= \frac{1}{|r|}\int_{\R} V_{n-1}(sT^n \cap H_{re_1,t}) d \mu(t)
\end{align*}
for every $s>0$ and $r \neq 0$.

Moreover, if $Z(T^n) \in \ACF{n}$, there is a locally integrable function $\zeta: \R \to \R$ such that
\begin{align*}
Z(sT^n) (re_1)
= \int_{sT^n} \zeta(r e_1 \cdot y) dy
\end{align*}
for every $s>0$ and $r \neq 0$.
Further, $Z(T^n) \in \CSF{n}$ gives $\zeta \in \CSF{}$ and the above relation also holds for $r =0$.
\end{lem}
\begin{proof}
Define the function $g(r) = r^nZ(T^n)(re_1)$ on $\ro{}$ and let $g^{(k)}$ denote the $k$-th derivative (if it exists) of $g$.
We assume first that $g$ has the following property: \\
($\ast$) $g^{(k)}$ exists and is continuous on $\ro{}$ from $k=1$ up to $n-1$ and
\begin{align*}
\lim_{r\to 0} g^{(k)}(r)=0
\end{align*}
for $0 \leq k \leq n-1$.

By $(\ast)$, there is a signed Radon measure $\mu$ on $\R$ such that
\begin{align}\label{eq:bv}
\int_0^r (r-t)^{n-1} d \mu(t) = (n-1)\int_0^r (r-t)^{n-2} g^{(n-1)}(t) dt.
\end{align}
Indeed, let first $r>0$. We define the function $j: \R \to \R$ by
\begin{align*}
j(t)=\begin{cases}
g^{(n-1)}(t), & t>0, \\
0, & t \leq 0.
\end{cases}
\end{align*}
Property $(\ast)$ gives that $j$ is locally a BV function.
By the structure theorem of BV functions \cite{MR3409135}, there is a signed Radon measure $\mu$ such that
\begin{align*}
\int_{\R} \theta(t) \mu(t) = -\int_{\R} j(t) \frac{\dif}{dt}\theta(t) dt
\end{align*}
for every $\theta \in C_c^1(\R)$, the set of compactly supported and continuously differentiable functions on $\R$.
Moreover, since $j$ is continuous and $j(t)=0$ for all $t \leq 0$, the measure $\mu$ is continuous and concentrated on $(0,\infty)$.
Hence, we can choose $\theta \in C_c^1(\R)$ such that $\theta(t)=(r-t)^{n-1} \kik{[o,r]}(t)$ for $t \geq 0$ and
\begin{align*}
\int_{\R} \theta(t) \mu(t) &= \int_{0}^r  (r-t)^{n-1} d\mu(t), \\
-\int_{\R} j(t) \frac{\dif}{dt}\theta(t) dt &= (n-1)\int_0^r (r-t)^{n-2} g^{(n-1)}(t) dt.
\end{align*}
Hence we get \eqref{eq:bv} for $r>0$.
For $r<0$, we get similarly a signed and continuous Radon measure $\nu$ concentrated on $(-\infty,0)$ satisfying \eqref{eq:bv}.
Now taking the sum of $\mu$ and $\nu$, \eqref{eq:bv} still holds.

Using integration by parts, combined with $(\ast)$ and \eqref{eq:bv}, we have
\begin{align*}
\frac{1}{(n-1)!}\int_0^r (r-t)^{n-1} d \mu(t)
&=\frac{1}{(n-2)!}\int_0^r (r-t)^{n-2} g^{(n-1)}(t) dt \\
&=\frac{1}{(n-3)!}\int_0^r (r-t)^{n-3} g^{(n-2)}(t) dt \\
&=\cdots =\int_0^r  g^{(1)}(t) dt=g(r)=r^nZ(T^n)(re_1)
\end{align*}
for any $r \neq 0$.
For any $x \in \ro{n}$, since the function $s \mapsto Z(s^{1/n}T^n)(s^{-1/n}x)$ is regular on $(0,\infty)$, Lemma \ref{lem:fe1} (b) shows that
\begin{align}\label{homco1}
Z (sT^{n}) (s^{-1}x) = s^n Z(T^n)(x)
\end{align}
for any $s>0$.
Hence
\begin{align*}
  Z (sT^{n}) (re_1) = s^n Z(T^n)(rs)&=\frac{1}{(n-1)!r^n}\int_0^{rs} (rs-t)^{n-1} d \mu(t) \\
  &=\frac{1}{(n-1)!r}\int_0^{rs} (s-t/r)^{n-1} d \mu(t) \\
  &=\frac{1}{r}\int_0^{rs} V_{n-1}(sT^n \cap H_{re_1,t}) d \mu(t) \\
  &=\frac{1}{|r|}\int_{-h_{sT}(-re_1)}^{h_{sT}(re_1)} V_{n-1}(sT^n \cap H_{re_1,t}) d \mu(t) \\
  &= \frac{1}{|r|}\int_{\R} V_{n-1}(sT^n \cap H_{re_1,t}) d \mu(t).
\end{align*}

Now we show that the function $g$ has the desired property $(\ast)$ if $ZT^n \in \CFr{n}$.
We claim that the function $r \mapsto r ZT^n(re_1+\dots+re_m)$ is differentiable on $\ro{}$ and
\begin{align}\label{eq:dif2}
&\frac{\dif (rZT^n(re_1+\dots+re_m))}{\dif r} \notag\\
&=m ZT^n(re_1+\dots+re_{m+1})-(m-1)ZT^n(re_1+\dots+re_{m})
\end{align}
for any $r \neq 0$ and $1\leq m \leq n-1$.
Indeed, set $y_k=te_3+\dots+te_{k+2}+re_{k+3}+\dots+re_{m+1}$ for integer $k \in [0,m-1]$.
By \eqref{eq:fe2}, \eqref{homco1}, \eqref{38} and \eqref{40},
\begin{align}\label{eq:dif1}
ZT^n(re_1+te_2+y_k)=\frac{r ZT^n(re_1+y_k)-tZT^n(te_1+y_k)}{r-t}
\end{align}
for any $r,t \neq 0$ such that $rt>0$.
Take the sum of \eqref{eq:dif1} over all $k$ from $0$ to $m-1$ and let $t \to r$.
By $ZT^n(\cdot) \in \CFo{n}$, we have
\begin{align*}
&mZT^n(re_1+\dots+re_{m+1})\\
&=\lim_{t \to r}\frac{r ZT^n(re_1+y_0)-tZT^n(te_1+y_{m-1})}{r-t} + \sum_{k=0}^{m-2}\frac{-tZT^n(te_1+y_k)+rZT^n(re_1+y_{k+1})}{r-t}.
\end{align*}
Together with the permutation invariance of $ZT^n$, we get \eqref{eq:dif2}.

Let $\tilde{g}_0(r)=ZT^n(re_1)$.
The relation \eqref{eq:dif2} shows that we can define functions $\tilde{g}_k$ on $\ro{}$ for $k=1,\dots n-1$ inductively by $\tilde{g}_{k}(r)=\frac{\dif}{\dif r}(r\tilde{g}_{k-1}(r))$ and $\tilde{g}_{k}(r)$ is a linear combination of $ZT^n(re_1),\dots,ZT^n(re_1+\dots+re_{k+1})$.
Using the chain rule formula inductively, $g^{(k)}(r)$ is a linear combination of $r^{n-k}\tilde{g}_0(r),\dots, r^{n-k}\tilde{g}_k(r)$ for $k=0,\dots n-1$.
Hence $g$ has the desired properties $(\ast)$ since $ZT^n \in \CFr{n}$.

For $ZT^n \in \ACF{n}$, the argument in the last paragraph shows that $g^{(n-1)}$ is absolutely continuous on any interval in $\R$ after setting $g^{(n-1)}(0)=0$.
Hence $\mu$ has a density $\zeta$. The desired result follows from Fubini's formula \eqref{eq:fubini}.
For $ZT^n \in \CSF{n}$, notice that now $g:r \mapsto r^nZ(T^n)(re_1) \in \CSF{}$.
We just need to set $\zeta(t)=g^{(n)}(t)$ for all $t \in \R$, and obtain $\zeta \in \CSF{}$.
\end{proof}

The following theorem follows directly from \eqref{eq:fe2}, \eqref{30}, Lemma \ref{lem:fe1} (c) and Lemma \ref{lemuq}.
\begin{thm}\label{thm:simgl0}
Let $n \ge 3$. A map $Z: \mpo{n} \to \FFo{n}$ is a simple and $\GL{+}{n}$ covariant valuation of weight $0$
if and only if
\begin{align*}
ZP(x)= 0
\end{align*}
for every $P \in \mpo{n}$ and $x \in \ro{n}$.
\end{thm}

\vskip 30pt

\section{On lower-dimensional sets}
In this section, we prove the following theorem.
\begin{thm}\label{thm:low}
Let $n \ge 3$.
If $Z: \mpo{n} \to \FFo{n}$ is an $\sln$ covariant valuation and the function $r \mapsto ZT^{n-1}(\pm re_n)$ is regular on $(0,\infty)$,
then there are functions $\zeta_1,\zeta_2 \in \FF{}$ and a constant $c_{n-1} \in \R$ such that
\begin{align}\label{eq:low2}
ZP(x)&= \zeta_1(h_P(x)) +\zeta_1^R(h_{-P}(x)) + \eutr{-}{\zeta_2}(P)(x) + \eutr{-}{\zeta_2^R}(-P)(x) + \frac{c_{n-1}}{(n-1)! |x|}V_{n-1}(P \cap H_{x,o})
\end{align}
for every $x \in \ro{n}$ and $P \in \mpo{n}$ with $\dim P <n$.

Moreover, if $ZP \in \CFo{n}$ for every $P \in \mpo{n}$, then $\zeta_1,\zeta_2 \in \CF{}$ and $c_{n-1}=0$.
\end{thm}

First, we show the following property of an $\sln$ covariant map.
\begin{lem}\label{lem1}
Let $n \geq 2$ and $Z : \mathcal{P}_o ^n \to \FFo{n}$ be an $\SLn$ covariant map. If the linear hull of $P \in \mpo{n}$ is $\lin \{e_1,\dots,e_d\}$ for some $0 \leq d < n$, then
\begin{align*}
  Z(sP) (x) =
\begin{cases}
ZP (s(r_1e_1 + \dots + r_de_d)), &r_1e_1 + \dots + r_de_d \neq o, \\
ZP (e_{n}),                    &r_1e_1 + \dots + r_de_d = o , d \leq n-2  , \\
ZP (s^{-(n-1)}r_n e_n),       & r_1e_1 + \dots + r_de_d = o , d=n-1
\end{cases}
\end{align*}
for every $s>0$ and $x = r_1e_1 + \dots + r_ne_n \in \ro{n}$.
\end{lem}
\begin{proof}
First let $s=1$.
Set $\phi := \left[ {\begin{array}{*{20}{c}}
I_d &A \\
0& B
\end{array}} \right] \in \SLn$, where $I_d \in \mathbb{R}^{d \times d}$ is the identity matrix, $A \in \mathbb{R}^{d \times (n-d)}$ is an arbitrary matrix, $0 \in \mathbb{R}^{(n-d)\times d}$ is the zero matrix, $B \in \mathbb{R}^{(n-d) \times (n-d)}$ is an arbitrary matrix with the determinant $1$.
Also, let $x'=r_1e_1 + \dots + r_de_d$ and $x''=x-x'$.
Combining with $\phi P = P$ and the $\SLn$ covariance of $Z$, we have
\begin{align*}
Z(P) (x) = Z(\phi P) (x) = Z(P) (\phi ^t x) = Z(P) \left( {\begin{array}{*{20}{c}}
{x'}\\
{A^t x' + B^t x''}
\end{array}} \right).
\end{align*}
If $x' \neq o$, we can choose appropriate matrices $A$ and $B$ such that $A^t x' + Bx''=0$.
Hence $ZP (x) = ZP (x')$ for $x' \neq 0$.
If $x'=o$, then $x'' \neq o$.
For $d \leq n-2$, we can choose an appropriate matrix $B$ such that $Bx''=e_n$ and then $Z(P)(x) = ZP(e_n)$ for every $x' =o$.
For $d=n-1$, the statement is trivial for $x'=o$.
Hence we confirmed the case $s=1$ for arbitrary $x \neq o$.

Let $s>0$. The above argument gives that
\begin{align*}
  ZP \left(sx'+s^{-d/(n-d)}x''\right)=\begin{cases}
    ZP (sx'), &x' \neq o, \\
    ZP (e_{n}), &x' = o , d \leq n-2  , \\
    ZP (s^{-(n-1)}r_n e_n), & x' = o , d=n-1
\end{cases}.
\end{align*}
Define $\psi \in \sln$ by $\psi e_i =se_i$ for $1 \leq i \leq d$ and $\psi e_j = s^{-d/(n-d)} e_j$ for $d+1 \leq j \leq n$.
Since $Z$ is $\SLn$ covariant,
we have
\begin{align*}
  Z(sP)(x)=Z(\psi P)(x)=ZP(\psi^t x)=ZP \left(sx'+s^{-d/(n-d)}x''\right)
\end{align*}
Hence the desired statement holds for any arbitrary $s>0$.
\end{proof}

Now we consider $\SLn$ covariant valuations.

\begin{lem}\label{lem:codim}
Let $Z: \mpo{n} \to \FFo{n}$ be an $\SLn$ covariant valuation and let the function $r \mapsto ZT^{n-1}(\pm re_n)$ be regular on $(0,\infty)$.
If $n\ge 3$, then there are $c_{n-1}, c \in \R$ such that
\begin{align*}
ZT^{n-1}(re_n) = c_{n-1} |r|^{-1} +c, ~~ZT^d(re_n)=c
\end{align*}
for any $r \neq 0$, $1 \leq d \leq n-2$.
If $n = 2$ and $Z$ is $\GL{}{2}$ covariant of weight $0$, then there is $c \in \R$ such that
\begin{align*}
ZT^{1}(re_2) = c
\end{align*}
for any $r \neq 0$.
\end{lem}
\begin{proof}
First let $n=2$.
For an arbitrary $r \neq 0$, we define $\phi \in \GL{}{2}$ by the following:
\begin{align*}
\phi e_{2}=re_{2} \text{~and~} \phi e_1=e_1.
\end{align*}
Since $Z$ is $\GL{}{2}$ covariant of weight $0$, we have
\begin{align*}
Z(T^{1})(e_2) = Z(\phi T^{1})(e_2) = Z(T^{1})(re_2),
\end{align*}
which implies the desired result.

Now let $n\ge 3$ and $r>0$.
Combining Lemma \ref{lem1} with \eqref{30} for $x=re_n$ and $d \leq n-2$, we have
$Z(T^{d}) (e_n)=Z(T^{d}) (re_n) = Z (\widehat{T}^{d-1}) (re_n)$.
Using the $\SLn$ covariance of $Z$, we have
\begin{align*}
Z(T^{n-2})(re_n)=\dots=Z(T^{1})(re_n)=:c.
\end{align*}
Now using \eqref{30} for $x=re_n$ and $d = n-1$ together with the $\SLn$ covariance of $Z$,
we have
\begin{align}\label{eq:vallow1}
ZT^{n-1}(re_n) + c = ZT^{n-1}(\lambda^{-1}re_n) + ZT^{n-1}((1-\lambda)^{-1}re_n)
\end{align}
for any $r > 0$.
Set $g(r)=ZT^{n-1}(r^{-1}e_n)$.
For arbitrary $a,b>0$, we can choose $r=(a+b)^{-1}$ and $\lambda=\frac{a}{a+b}$ in \eqref{eq:vallow1}.
Hence
\begin{align*}
g(a+b)+c = g(a)+g(b)
\end{align*}
for any $a,b>0$.
Now the regular function $g-c$ satisfying the Cauchy functional equation implies that
\begin{align*}
ZT^{n-1}(re_n) = c_{n-1}r^{-1} +c
\end{align*}
for any $r>0$, where $c_{n-1}$ is a constant.

Since $Z$ is $\sln$ covariant, we have $ZT^{d}(re_n)=ZT^{d}(-re_n)$ for every $d \leq n-1$ if $n \ge 3$,
which confirms the case $r<0$.
\end{proof}

\begin{lem}\label{lem:unique1}
For $n\ge 3$, let $Z: \mpo{n} \to \FFo{n}$ be an $\SLn$ covariant valuation and let the function $r \mapsto ZT^{n-1}(\pm re_n)$ be regular on $(0,\infty)$.
If \begin{align}\label{asm:uni}
Z\{o\}(e_n)=ZT^1(re_1)=ZT^2(re_1)=ZT^{n-1}(re_n)=0
\end{align}
for every $r \neq 0$, then $Z$ is simple.
\end{lem}
\begin{proof}
By Lemma \ref{lemuq} and Lemma \ref{lem1}, we only need to show that $Z(T^d)(x)=0$ for any $x \in \ro{n}$ and $0 \leq d \leq n-1$. Lemma \ref{lem1} together with $Z\{o\}(e_n)=0$ implies that it holds for $d=0$.

Lemma \ref{lem:codim} implies that $ZT^{n-1}(re_n)=c_{n-1}|r|^{-1}+c$ and $ZT^k(re_n)=c$ for any $r \neq 0$ and $1\leq k\leq n-2$.
Combined it with the assumption that $ZT^{n-1}(re_n)=0$ for any $r \in \R$, we have
\begin{align}\label{eq:unq2}
ZT^d(re_n)=0
\end{align}
for all $1 \leq d \leq n-1$ and $r \neq 0$.
Hence, together with \eqref{asm:uni}, Lemma \ref{lem1} implies that our statement holds for $d=1$.

Now let $2 \leq d \leq n-1$. Assume that our statement holds for $d-1$.

Set
\begin{align}\label{eq:fe1}
f(s;r_1e_1+\dots +r_de_d)=Z(sT^{d})(r_1e_1+\dots +r_de_d).
\end{align}
Together with the induction assumption, \eqref{30} and the $\SLn$ covariance of $Z$, we get the equation \eqref{30a}.
Also, Lemma \ref{lem1} implies that $f(s^{1/n};s^{-1/n}x)=f(x)$ for any $s>0$, $x \in \ro{d}$ and $0 \leq d \leq n-1$.
Now Lemma \ref{lem:fe1} (c), together with the assumption $ZT^2(re_1)=0$ (only needed for $d=2$) gives that $ZT^d(x)=0$ for $x \in \ro{d}$.
Generally, for $x \in \ro{n}$, we write $x=x'+x''$ as in the proof of Lemma \ref{lem1}.
By \eqref{eq:unq2}, we have $Z(T^d)(re_n)=0$ for all $r \in \ro{}$.
Together with Lemma \ref{lem1}, we get $Z(sT^d)(x)=0$ for $x \in \ro{n}$.
\end{proof}

\begin{proof}[Proof of Theorem \ref{thm:low}]
By Lemma \ref{lem:codim}, there is a $c_{n-1}$ such that
\begin{align*}
ZT^{n-1}(re_n)=c_{n-1}|r|^{-1}+ZT^1(e_n).
\end{align*}
for any $r \neq 0$.
Define $\zeta_1,\zeta_2$ by
\begin{align}
&\zeta_1(r)=\begin{cases}
ZT^2(re_1)-\frac{1}{2}ZT^1(e_n), & r \neq 0,\\
\frac{1}{2}ZT^1(e_n), & r=0,
\end{cases} \label{zeta1} \\
&\zeta_2(r)=\begin{cases}
ZT^2(re_1)-ZT^1(re_1)+\frac{1}{2}(Z\{o\}(e_n)-ZT^1(e_n)), & r \neq 0,  \\
\frac{1}{2}(Z\{o\}(e_n)-ZT^1(e_n)), & r=0.
\end{cases} \label{zeta2}
\end{align}

Now we consider a new valuation which is the difference of $Z$ and the valuation given in Theorem \ref{thm:low} (with $\zeta_1,\zeta_2$ and $c_{n-1}$ defined above).
Elementary calculations show that the new valuation satisfies \eqref{asm:uni}.
Hence the statement of the first part follows from Lemma \ref{lem:unique1}.

If $ZT^1,ZT^2 \in \CFo{n}$, the definitions of $\zeta_1,\zeta_2$ show that $\zeta_1,\zeta_2 \in \CFo{}$.
Let $-1 < r <1$. The statement of the first part shows that
\begin{align*}
Z[o,e_2](-e_1+re_2)&=\zeta_1(r)+\zeta_1(0)-\zeta_2(r)+\zeta_2(0), \\
Z[0,e_1,e_2](-e_1+re_2)&=\begin{cases}
\zeta_1(r)+\zeta_1(-1),  & 0 \le r <1, \\
\zeta_1(0)+\zeta_1(-1)-\zeta_2(r)+\zeta_2(0), &-1<r<0.
\end{cases}
\end{align*}
Since $Z[o,e_2](-e_1+re_2) \to Z[o,e_2](-e_1)$ and $Z[0,e_1,e_2](-e_1+re_2) \to Z[0,e_1,e_2](-e_1)$ as $r \to 0$, we find that $\zeta_1,\zeta_2$ are continuous at $0$.
Hence $\zeta_1,\zeta_2 \in C(\R)$.

Now, only the function $x \mapsto \frac{1}{|x|}V_{n-1}(T^{n-1} \cap H_{x,o})$ is not continuous on $\ro{n}$.
Thus $c_{n-1}=0$.
\end{proof}

\section{Proofs of Theorems \ref{mthm:sln3}-\ref{thm:co0}}\label{sec:mainpr}

The ``if" parts of Theorems \ref{mthm:sln3}-\ref{thm:co0} are already completed by \S \ref{sec:val}.
Hence we only deal with the ``only if" parts.

\begin{proof}[Proof of Theorem \ref{mthm:sln}]
Since $ZT^{n-1} \in \CFo{n}$, the function $r \mapsto ZT^{n-1}(\pm re_n)$ is regular on $(0,\infty)$.
Hence Theorem \ref{thm:low} gives
\begin{align}\label{eq:low}
ZP(x)&= \zeta_1(h_P(x)) +\zeta_1^R(h_{-P}(x)) + \eutr{-}{\zeta_2}(P)(x) + \eutr{-}{\zeta_2^R}(-P)(x)
\end{align}
for every $x \in \R^n$ and $P \in \mpo{n}$ with $\dim P <n$, where $\zeta_1,\zeta_2 \in \CF{}$.
Let $\tilde{Z}P(x)=ZP(x)-(\zeta_1(h_P(x)) +\zeta_1^R(h_{-P}(x)) + \eutr{-}{\zeta_2}(P)(x) + \eutr{-}{\zeta_2^R}(-P)(x))$.
Clearly $\tilde{Z}: \mpo{n} \to \CFr{n}$ is a regular, simple and $\sln$ covariant valuation.
Now the desired result follows from Theorem \ref{mthm:simpsln}.
\end{proof}

Before proving Theorem \ref{thm:co0}, we first show the following.
\begin{lem}\label{lem:unique2}
Let $n=2$ and $Z: \mpo{2} \to \FFo{2}$ be a $\GL{}{2}$ covariant valuation of weight $0$.
If \begin{align*}
Z\{o\}(e_2)=ZT^1(re_1)=ZT^2(re_1)=ZT^{1}(e_2)=0
\end{align*}
for every $r \neq 0$, then $Z=0$.
\end{lem}
\begin{proof}
Since we assume that $Z$ is $\GL{}{2}$ covariant of weight $0$, together with Lemma \ref{lem:codim} (for $n=2$), the proof is the same as the proof of Lemma \ref{lem:unique1} (for $n=3$).
\end{proof}

\begin{proof}[Proof of Theorem \ref{thm:co0}]
Let first $n \ge 3$.
Since $Z$ is $\GL{+}{n}$ covariant of weight $0$, $r \mapsto ZT^{n-1}(\pm re_n)=ZT^{n-1}(e_1)$ is regular on $(0,\infty)$.
Hence by Theorem \ref{thm:low}, \eqref{eq:low2} holds.
Only the valuation $P \mapsto \frac{1}{|x|}V_{n-1}(P \cap H_{x,o})$ is not $\GL{+}{n}$ covariant of weight $0$.
Similar to the proof of Theorem \ref{mthm:sln}, the desired result follows from Theorem \ref{thm:simgl0}.

Let $n=2$. By Lemma \ref{lem:unique2}, the argument is similar to the proof of Theorem \ref{thm:low}.
\end{proof}

The following lemma will be used in the proof of Theorem \ref{mthm:sln3}.
\begin{lem}\label{lem:contin2}
Let $\zeta_1,\zeta_2 \in \CF{}$ and $Z:\mpo{n} \to \FFo{n}$ defined by \eqref{eq:low}
for every $x \in \ro{n}$ and $P \in \mpo{n}$ with $\dim P<n$.
If $Z$ is continuous on $\mpo{n}$, then
\begin{align*}
 \zeta_2 \equiv 0.
\end{align*}
\end{lem}
\begin{proof}
Let $r \ge 0$, $s,t>0$ and $P_{r}=[-te_1,e_1]+[-e_2,re_2]$.
By \eqref{eq:low}, we have
\begin{align*}
ZP_{r}(se_1)&= \zeta_1(s) + \zeta_1(-st)
+ \zeta_2(s)+\zeta_2(-st)
\end{align*}
for $r>0$ and
\begin{align*}
ZP_{0}(se_1)&= \zeta_1(s) + \zeta_1(-st)
\end{align*}
for $t>0$. Since $ZP_{r}(se_1) \mapsto ZP_0(se_1)$ for any $s>0$, we have
\begin{align}\label{eq:contin1}
\zeta_2(s)+ \zeta_2(-st)=0
\end{align}
for any $s,t>0$.
Since $\zeta_2\in \CF{}$, taking $s \mapsto 0$ in \eqref{eq:contin1}, we have $\zeta_2(0)=0$.
Next, taking $t \mapsto 0$, we have $\zeta_2(s)=0$ for any $s>0$.
Back to \eqref{eq:contin1}, we get $\zeta_2(s)=0$ for any $s<0$.
\end{proof}

\begin{proof}[Proof of Theorem \ref{mthm:sln3}]
Lemma \ref{lem:contin} shows that $ZP \in \CFo{n}$ for every $P \in \mpo{n}$.
Hence the same argument as in the proof of Theorem \ref{mthm:sln} shows that \eqref{eq:low} holds for $\zeta_1,\zeta_2 \in \CF{}$.
Together with the continuity of $Z$, by Lemma \ref{lem:contin2}, we have $\zeta_2 \equiv 0$.

Now it follows from Lemma \ref{lem:suppf} and Lemma \ref{lem:ref} that the map $\bar Z :P \mapsto ZP(x)-(\zeta_1(h_P(x)) +\zeta_1^R(h_{-P}(x))$ is a continuous, simple and $\sln$ covariant valuation.
By the remark after Theorem \ref{mthm:simpsln}, we only need to show that $s \mapsto \bar Z(sT^n)(s^{-1}x)$ is regular on $(0,\infty)$ and $\lim_{s \to 0} s \bar Z(T^n)(sx)=0$ for any $x \in \ro{n}$.
The last follows from Lemma \ref{lem:fe1} (b) (that \eqref{homco1} holds) and Lemma \ref{lem:contin}.
\end{proof}

\section{Euler-type relations}\label{sec:euler}
The maps $\eutr{\pm}{\zeta},\eutr{}{\zeta}$ are all $\GL{}{n}$ covariant valuations of weight $0$; see Lemma \ref{lem:val1}.
This fact together with Theorem \ref{thm:co0} gives a new proof of the following Euler-type relation (for all $\MP^n$).

\begin{mcor}\label{co:eusupp}
For $n \ge 0$, the Euler-type relation
\begin{align}\label{eq:eusupp}
\sum_{F \in \face{}{}(P)}(-1)^{\dim F}\zeta \left(h_{F}(x)\right)= \zeta(-h_{-P}(x))
\end{align}
holds for every $P \in \MP^n$, $\zeta \in \FF{}$ and $x \in \R^n$.
\end{mcor}

Shephard \cite{MR232280} established this Euler-type relation for linear functions $\zeta$.
His proof also works for arbitrary $\zeta$.
For linear $\zeta$, McMullen \cite{McM77} established Euler-type relations for an arbitrary function-valued valuation $Z$ with the additional assumption that $Z(P+y)(x)=ZP(x)+v(P,x) \cdot y$ for all $y\in \R^n$ where $v(P,x)$ is any vector.

Notice that \eqref{eq:eusupp} for $\zeta \equiv 1$ is the Euler-Schl\"{a}fli-Poincar\'{e} formula
\begin{align*}
\sum_{F \in \face{}{} (P)} (-1)^{\dim F}V_0(F) = \sum_{j=0}^{\dim P} (-1)^{j}|\face{}{j}(P)| =V_0(P),
\end{align*}
where $V_0$ is the Euler characteristic, which is $1$ for any convex body and $0$ for the empty set.

\begin{proof}[Proof of Corollary \ref{co:eusupp}]
The definition of $\eutr{}{\zeta}=\sum_{F \in \face{}{}(P)}(-1)^{\dim F}\zeta \left(h_{F}\right)$ does not rely on the dimension of the ambient space, so we can set the ambient space as $\R^n$ for $n \ge 2$.

For $x =o$, \eqref{eq:eusupp} is either the trivial case $0=0$ or the Euler-Schl\"{a}fli-Poincar\'{e} formula.
The last can be also obtained by setting $\zeta \equiv 1$ and $x \in \ro{n}$ in \eqref{eq:eusupp}.
Hence we only need to prove \eqref{eq:eusupp} for $x \in \ro{n}$.

Lemma \ref{lem:val1} shows that $\eutr{}{\zeta}: \mpo{n} \to \MF(\R^n)$ is a $\GL{}{n}$ covariant valuation of weight $0$.
Hence $\eutr{}{\zeta}$ is described in Theorem \ref{thm:co0}, that is,
\begin{align}\label{eq8091}
  \eutr{}{\zeta} P(x)=\zeta_1(h_P(x)) +\zeta_1^R(h_{-P}(x)) + \eutr{-}{\zeta_2}(P)(x) + \eutr{-}{\zeta_2^R}(-P)(x).
\end{align}
with some $\zeta_1,\zeta_2 \in \FF{}$ for all $P \in \mpo{n}$.
By the definition of $\eutr{}{\zeta}$, we get that
\begin{align*}
\eutr{}{\zeta} T^2(re_1)=\eutr{}{\zeta} T^1 (re_1)=
\begin{cases} \zeta(0), &r > 0,
\\ \zeta(r), &r<0, \end{cases}
\end{align*}
and
\begin{align*}
\eutr{}{\zeta} T^1(e_n)=\eutr{}{\zeta} \{o\} (e_n) =\zeta(0).
\end{align*}
Since $\zeta_1,\zeta_2$ are determined by \eqref{zeta1}, \eqref{zeta2} for $Z=\eutr{}{\zeta}$, we have
\begin{align*}
\zeta_1(r)=\begin{cases}
\frac{1}{2}\zeta(0), & r \ge 0,\\
\zeta(r) - \frac{1}{2}\zeta(0), & r <0,
\end{cases}
\end{align*}
and $\zeta_2 \equiv 0$.
For $o \in P$, we have $h_P(x),h_{-P}(x) \ge 0$.
Hence \eqref{eq8091} implies that \eqref{eq:eusupp} holds for $o \in P$.

Let $x_1,\dots,x_m$ be the vertices of $P \in \mpo{n}$ and let $x \in \ro{n}$.
By the definition of support functions,
we have $h_F(x)=x \cdot x_{j(F)}$ for all $F \in \face{}{}(P)$ and $-h_{-P}(x)=x \cdot x_{\tilde j(P)}$ with some $1 \leq j(F),\tilde j(P) \leq m$.
Hence we can reformulate \eqref{eq:eusupp} for these $P$ and $x$ as
\begin{align}\label{eq:eus1}
0=\zeta(-h_{-P}(x))-\sum_{F \in \face{}{}(P)}(-1)^{\dim F}\zeta \left(h_{F}(x)\right) =\sum_{i=1}^m a_i \zeta(x \cdot x_i),
\end{align}
where
\begin{align}\label{eq:loceu1}
  a_i=\kik{i}(\tilde j(P)) - \sum_{F \in \face{}{}(P)}(-1)^{\dim F} \kik{i}(j(F)).
\end{align}
Now we claim that $a_i=0$ for every $1\leq i \leq m$.
Indeed, assume first that $x \cdot x_i$ are all different.
Since $\R$ is an infinite dimensional space over the rational field $\Q$, we can find $r_1,\dots,r_m \in \ro{}$ such that they are $\Q$-independent.
Also, $a_i$ are integers and do not rely on the choice of $\zeta$, so we choose $\zeta \in \FF{}$ such that $\zeta(x \cdot x_i)=r_i$.
Thus \eqref{eq:eus1} gives that all $a_i=0$.
If some of $x \cdot x_i$ are the same, we just need that the (different) elements of the set $\{\zeta(x \cdot x_i)\}$ to be nonzero $\Q$-independent numbers.

The last part is to show that \eqref{eq:eusupp} holds for $P+y$ for all $y \in \R$.
The vertices of $P+y$ are $x_1+y,\dots,x_m+y$.
Notice that $h_{F+y}(x)=x \cdot (x_{j(F)} + y)$ and $-h_{-(P+y)}(x)=x \cdot (x_{\tilde j(P)} + y)$.
We have
\begin{align*}
  \zeta(-h_{-(P+y)}(x))-\sum_{F \in \face{}{}(P+y)}(-1)^{\dim F}\zeta \left(h_{F}(x)\right) = \sum_{i=1}^m a_i \zeta(x \cdot (x_i+y)).
\end{align*}
Hence \eqref{eq:eusupp} holds for all $P+y$.
Since $P$ and $x$ are chosen arbitrary in the above argument, the proof is completed.
\end{proof}

Corollary \ref{co:eusupp} immediately implies the following.
\begin{mcor}\label{co:eumixv}
For $n \ge 1$, we have
\begin{align*}
\sum_{F \in \face{}{}(P)}(-1)^{\dim F}\int_{\Sn} \zeta \left(\frac{h_{F}(u)}{h_L(u)}\right)dV_L(u)=\int_{\Sn} \zeta \left(-\frac{h_{-P}(u)}{h_L(u)}\right)dV_L(u)
\end{align*}
for every $\zeta \in \FF{}$, $P \in \MP^n$ and $L \in \MK^n$ such that $o$ is contained in the interior of $L$.
\end{mcor}
Here $V_L$ is the cone volume measure of $L$, that is, $V_L(\omega) = \frac{1}{n}\int_{\nu_K^{-1}(\omega)} h_K(\nu_K(x))d\mathcal{H}^{n-1}(x)$ for any Borel set $\omega \subset \Sn$. Here $\nu_K$ is the Gauss map of the boundary of $K$ and $\mathcal{H}^{n-1}$ is the $(n-1)$-dimensional Hausdorff measure.

We also obtain local versions of the Euler-Schl\"{a}fli-Poincar\'{e} formula related to $\eutr{\pm}{\zeta}(P)$.
Let $1$ denote the constant function $1$ on $\R$.
By the definition, $\eutr{\pm}{1}(P)=\sum_{j=0}^{\dim P} (-1)^{j} |\face{\pm}{j}(P)|$.
If further $o \in P$, then $\eutr{-}{1}(P)=\sum_{F \in \face{}{} (P)} (-1)^{\dim F}V_0(o \cap F)$.
We obtain the following result.

\begin{mcor}[The local Euler-Schl\"{a}fli-Poincar\'{e} formula]\label{co:leu}
For $n \ge 0$, we have
\begin{align*}
&\sum_{j=0}^{\dim P} (-1)^{j}|\face{-}{j}(P)|=(-1)^{\dim P}V_0(o \cap \relint P), \\
&\sum_{j=0}^{\dim P} (-1)^{j}|\face{+}{j}(P)|=V_0(o \cap P),
\end{align*}
and
\begin{align*}
\sum_{F \in \face{}{} (P)} (-1)^{\dim F}V_0(x \cap F)=(-1)^{\dim P}V_0(x \cap \relint P)
\end{align*}
for every $x \in \R^n$ and $P \in \MP^n$.
\end{mcor}

The last formula was also established by Kabluchko, Last, and Zaporozhets \cite{MR3679943}.
They even established a general formula, replacing $x$ by an arbitrary affine subspace (the right side needs to be changed a little bit in the general case).
Using such general formula, they confirmed the so-called Cowan formula (identity).

To prove Corollary \ref{co:leu}, we need introduce the following lemma.

\begin{lem}[Ludwig and Reitzner \cite{LR2017sl}]\label{co:eu}
Let $n \ge 2$. A map $Z: \MP^n \to \R$ is a $\GL{}{n}$ invariant valuation of weight $0$ if and only if
there are constants $c_1,c_2,c_3 \in \R$
such that
\begin{align*}
ZP=c_1V_0(P)+c_2(-1)^{\dim P}V_0(o \cap \relint P)+c_3 V_0(o \cap P)
\end{align*}
for every $P \in \MP^n$.
\end{lem}

\begin{proof}[Proof of Corollary \ref{co:leu}]
Since desired formulae do not rely on the choice of the dimension of the ambient space, we can set the ambient space as $\R^n$ for $n \ge 2$.
By Lemma \ref{lem:val1}, the left sides of the first two desired formulae are $\GL{}{n}$ invariant valuations of weight $0$.
Together with Lemma \ref{co:eu},
we have
\begin{align*}
\sum_{j=0}^{\dim P} (-1)^{j}|\face{\pm}{j}(P)|=c_1^{\pm} V_0(P)+c_2^{\pm}(-1)^{\dim P}V_0(o \cap \relint P) +c_3^{\pm} V_0(o \cap P).
\end{align*}
with $c_1^\pm,c_2^\pm,c_3^\pm$.
Choosing $P$ as $\{o\}$, $[-e_1,e_1]$ or $[e_1,2e_1]$,
we get
\begin{align*}
\begin{cases}
  c_1^{-}+c_2^{-}+c_3^{-}=1, \\
  c_1^{-}-c_2^{-}+c_3^{-}=-1, \\
  c_1^{-}=0,
\end{cases}
~~~~\begin{cases}
c_1^{+}+c_2^{+}+c_3^{+}=1,\\
c_1^{+}-c_2^{+}+c_3^{+}=1,\\
c_1^+=0.
\end{cases}
\end{align*}
Thus,
\begin{align*}
&c_1^- = 0 ,~c_2^-=1,~c_3^-=0, \\
&c_1^+ = 0 ,~c_2^+=0,~c_3^+=1,
\end{align*}
Hence the first two desired formulae are true.
For $o \in P$, the first formula is equivalent to the following
\begin{align}\label{eq:ESP}
\sum_{F \in \face{}{} (P)} (-1)^{\dim F}V_0(o \cap F)= (-1)^{\dim P}V_0(o \cap \relint P).
\end{align}
Clearly, \eqref{eq:ESP} also holds for $o \notin P$.
Now replacing $P$ by $P-x$ in \eqref{eq:ESP}, we get the third formula.
\end{proof}

\section{Valuations on $\MP^n$ and $\MK^n$}\label{sec:genpoly}
To extend Theorem \ref{mthm:sln} to $\MP^n$, the natural idea is to check whether the representation
\begin{align*}
&\zeta_1(h_P(x)) +\zeta_1^R(h_{-P}(x)) + \eutr{-}{\zeta_2}(P)(x) + \eutr{-}{\zeta_2^R}(-P)(x) +\frac{1}{|x|} \int_{\R} V_{n-1}(P \cap H_{x,t}) d\mu(t) \\
&\qquad +\tilde{\zeta}_1(h_{[P,o]}(x)) +\tilde{\zeta}_1^R(h_{-[P,o]}(x)) + \eutr{-}{\tilde{\zeta}_2}([P,o])(x) + \eutr{-}{\tilde{\zeta}_2^R}(-[P,o])(x) \\
&\qquad +\frac{1}{|x|} \int_{\R} V_{n-1}([P,o] \cap H_{x,t}) d\tilde{\mu}(t)
\end{align*}
gives all $\sln$ covariant $\CFr{n}$-valued valuations, where $\zeta_1,\zeta_2,\tilde{\zeta}_1,\tilde{\zeta}_2 \in \CF{}$ and $\mu,\tilde{\mu} \in \cms$.
Unfortunately, it does not.
Instead, replacing $\zeta^R(h_{-P})$ by $\eutr{+}{\zeta}(P)$ (by Corollary \ref{co:eusupp}) in Theorem \ref{mthm:sln} gives the idea of the following result.

\begin{mthm}\label{mthm:sln2}
Let $n \ge 3$. A map $Z: \MP^n \to \CFr{n}$ is a regular and $\sln$ covariant valuation
if and only if
there are $\zeta_1,\zeta_2,\tilde{\zeta}_1,\tilde{\zeta}_2 \in \CF{}$ and $\mu,\tilde{\mu} \in \cms$ such that
\begin{align}\label{val:sl3}
&ZP(x) \notag \\
&=\eutr{+}{\zeta_1}(P)(x) +\eutr{+}{\zeta_1^R}(-P)(x) + \eutr{-}{\zeta_2}(P)(x) + \eutr{-}{\zeta_2^R}(-P)(x) + \frac{1}{|x|} \int_{\R} V_{n-1}(P \cap H_{x,t}) d\mu(t)
  \notag \\
&\qquad +\eutr{+}{\tilde{\zeta}_1}([P,o])(x) + \eutr{+}{\tilde{\zeta}_1^R}(-[P,o])(x) + \eutr{-}{\tilde{\zeta}_2}([P,o])(x) + \eutr{-}{\tilde{\zeta}_2^R}(-[P,o])(x) \notag \\ &\qquad +\frac{1}{|x|} \int_{\R} V_{n-1}([P,o] \cap H_{x,t}) d\tilde{\mu}(t)
\end{align}
for every $P \in \MP^n$ and $x \in \ro{n}$.
\end{mthm}

\begin{proof}
Let $\MTon$ be the set of simplices in $\R^n$ with one vertex at the origin.
In the proof of Theorem \ref{mthm:sln}, we have shown that a regular and $\sln$ covariant valuation $Z: \MTon \to \CFr{n}$ can be described as
\begin{align*}
  Z(sT^d)(x)&=\eta_1(h_P(x)) +\eta_1^R(h_{-P}(x)) + \eutr{-}{\eta_2}(sT^d)(x) + \eutr{-}{\eta_2^R}(-sT^d)(x) \notag \\ &\qquad +\frac{1}{|x|} \int_{\R} V_{n-1}(sT^d \cap H_{x,t}) d\nu(t)
\end{align*}
for any $s>0$, $0 \leq d \leq n$ and $x \in \ro{n}$ with some $\eta_1,\eta_2 \in \CF{}$ and $\nu \in \cms$.
Further by Corollary \ref{co:eusupp} and $\eutr{+}{\eta}P(x)=\sum_{F \in \face{}{}(P)}(-1)^{\dim F}\eta \left(h_{F}(x)\right)$ when $o \in P$,
we get
\begin{align}\label{val:sl}
Z(sT^d)(x)&= \eutr{+}{\eta_1}(sT^d)(x) +\eutr{+}{\eta_1^R}(-sT^d)(x) + \eutr{-}{\eta_2}(sT^d)(x) + \eutr{-}{\eta_2^R}(-sT^d)(x) \notag \\ &\qquad +\frac{1}{|x|} \int_{\R} V_{n-1}(sT^d \cap H_{x,t}) d\nu(t).
\end{align}

For any $T \in \MTon \setminus \{o\}$, we write $\tilde{T}$ for its $(\dim T-1)$-dimensional face opposite to the origin.
We define the new map $\tilde{Z}: \MTon \to \CFr{n}$ by $\tilde{Z} T = Z \tilde{T}$ for any $T \in \MTon \setminus \{o\}$ and $\tilde{Z}\{o\}=0$.
It is not hard to show that $\tilde{Z}$ is also a regular and $\SLn$ covariant valuation on $\MTon$.
Hence, there are $\tilde{\eta}_1,\tilde{\eta}_2 \in \CF{}$ and $\tilde{\nu} \in \cms$ such that
\begin{align}\label{val:sl2}
Z(s[e_1,\dots,e_d])(x) &=\tilde{Z}(sT^d) \notag\\
&= \eutr{+}{\tilde{\eta}_1}(sT^d)(x) +\eutr{+}{(\tilde{\eta}_1)^R}(-sT^d)(x) + \eutr{-}{\tilde{\eta}_2}(sT^d)(x) + \eutr{-}{(\tilde{\eta}_2)^R}(-sT^d)(x) \notag \\
& \qquad +\frac{1}{|x|} \int_{\R} V_{n-1}(sT^d \cap H_{x,t}) d\tilde{\nu}(t)
\end{align}
for any $s>0$, $1 \leq d \leq n$ and $x \in \ro{n}$.
Set $\tilde{\zeta}_i=\tilde{\eta}_i$ and $\zeta_i=\eta_i-\tilde{\eta}_i$ for $i=1,2$ and $\tilde{\mu}=\tilde{\nu},~\mu=\nu-\tilde{\nu}$.
The valuation defined by the right side of \eqref{val:sl3} satisfies \eqref{val:sl} and \eqref{val:sl2}.
Now the proof is completed by Lemma \ref{lemuq2}.
\end{proof}

The extension of Theorem \ref{thm:co0} to $\MP^n$ is similar to Theorem \ref{mthm:sln2}.
\begin{mthm}\label{thm2:co0}
Let $n \ge 3$. A map $Z: \MP^n \to \FFo{n}$ is a $\GL{+}{n}$ covariant valuation
if and only if
there are $\zeta_1,\zeta_2,\tilde{\zeta}_1,\tilde{\zeta}_2 \in \FF{}$ such that
\begin{align*}
ZP(x) &=\eutr{+}{\zeta_1}(P)(x) +\eutr{+}{\zeta_1^R}(-P)(x) + \eutr{-}{\zeta_2}(P)(x) + \eutr{-}{\zeta_2^R}(-P)(x)
  \\
&\qquad +\eutr{+}{\tilde{\zeta}_1}([P,o])(x) + \eutr{+}{\tilde{\zeta}_1^R}(-[P,o])(x) + \eutr{-}{\tilde{\zeta}_2}([P,o])(x) + \eutr{-}{\tilde{\zeta}_2^R}(-[P,o])(x)
\end{align*}
for every $P \in \MP^n$ and $x \in \ro{n}$.

A map $Z: \MP^2 \to \FFo{2}$ is a $\GL{}{2}$ covariant valuation of weight $0$ if and only if the above relation holds for $n=2$.
\end{mthm}

Some readers may wonder why the valuation $\zeta(h_P)$ seems to be missing in Theorem \ref{mthm:sln2}.
In fact, it is included in the function classified in Theorem \ref{mthm:sln2} as
\begin{align}\label{eq:eu4}
\eutr{-}{\zeta}P-\eutr{-}{\zeta}[P,o]=\zeta(-h_{-P})-\zeta(-h_{[-P,o]})
\end{align}
This formula will be proved in the Theorem \ref{mthm:sln4}.

Analogs to Lemma \ref{lem:contin2}, the following lemma will be used in the proof of Theorem \ref{mthm:sln4}.

\begin{lem}\label{lem:contin3}
Let $\zeta_1,\tilde \zeta_1,\zeta_2,\tilde \zeta_2 \in \CF{}$ and $Z:\MP^n \to \FFo{n}$ defined by
\begin{align}\label{eq8041}
ZP(x)&=\eutr{+}{\zeta_1}(P)(x) +\eutr{+}{\zeta_1^R}(-P)(x) + \eutr{-}{\zeta_2}(P)(x) + \eutr{-}{\zeta_2^R}(-P)(x)
  \notag \\
&\qquad +\eutr{+}{\tilde{\zeta}_1}([P,o])(x) + \eutr{+}{\tilde{\zeta}_1^R}(-[P,o])(x) + \eutr{-}{\tilde{\zeta}_2}([P,o])(x) + \eutr{-}{\tilde{\zeta}_2^R}(-[P,o])(x)
\end{align}
for every $x \in \ro{n}$, $P \in \mpo{n}$ with $\dim P<n$ and for every $P \in \MP^n$ with $\dim P<n-1$.
If $Z$ is continuous, then
\begin{align*}
 \zeta_2+\tilde{\zeta}_2 \equiv 0,~\zeta_1 \equiv 0.
\end{align*}
\end{lem}
\begin{proof}
By \eqref{defval3}, Corollary \ref{co:eusupp}, the description \eqref{eq8041} shows that
\begin{align*}
  ZP(x)&=(\zeta_1 + \tilde \zeta_1)(h_P(x)) + (\zeta_2 + \tilde \zeta_2)(h_{-P}(x))
  + \eutr{-}{\zeta_2 + \tilde \zeta_2 }(P)(x) + \eutr{-}{(\zeta_2 + \tilde \zeta_2)^R}(-P)(x)
\end{align*}
for every $x \in \ro{n}$ and $P \in \mpo{n}$ with $\dim P<n$.
Hence the relation $\zeta_2+\tilde{\zeta}_2 \equiv 0$ follows from Lemma \ref{lem:contin2}.

Let $r \ge 0$, $s,t>0$ and $P_{r}=[-te_1,e_1]+re_2$.
The description \eqref{eq8041} gives that
\begin{align*}
Z P_r(se_1)= \tilde{\zeta}_1(s)+ \tilde{\zeta}_1(-st)
\end{align*}
for $r>0$, and
\begin{align*}
Z P_0(se_1)= \tilde{\zeta}_1(s)+\tilde{\zeta}_1(-st)+\zeta_1(s)+\zeta_1(-st).
\end{align*}
Similar to the proof in Lemma \ref{lem:contin2}, we have $\zeta_1 \equiv 0$.
\end{proof}

\begin{mthm}\label{mthm:sln4}
Let $n \ge 3$. A map $Z: \MP^n \to \FFo{n}$ is a continuous and $\sln$ covariant valuation
if and only if
there are $\zeta,\tilde{\zeta} \in \CF{}$ and $\mu,\tilde{\mu} \in \cms$ such that
\begin{align*}
ZP(x)&= \zeta(h_P(x)) +\zeta(-h_{-P}(x)) + \frac{1}{|x|} \int_{\R} V_{n-1}(P \cap H_{x,t}) d\mu(t)
\\
    &\qquad + \tilde{\zeta}(h_{[P,o]}(x)) +\tilde{\zeta}(-h_{-[P,o]}(x)) + \frac{1}{|x|} \int_{\R} V_{n-1}([P,o] \cap H_{x,t}) d\tilde{\mu}(t)
\end{align*}
for every $P \in \MP^n$ and $x \in \ro{n}$.
\end{mthm}
\begin{proof}
We first prove that \eqref{eq:eu4} holds for every $P \in \MP^n$ and $\zeta \in \FF{}$.
By \eqref{defval3}, \eqref{defval4} and \eqref{eq:eusupp}, we only need to show that
\begin{align*}
\face{}{}(P) \setminus \face{-}{}(P)= \face{}{}([P,o]) \setminus \face{-}{}([P,o]).
\end{align*}
Let $F \in \face{}{}(P) \setminus \face{-}{}(P)$ first.
There is a $u \in \relint N(P,F)$ such that $h_P(u)>0$.
Hence $P \subset H_{u,h_P(u)}^-$, $o \in \relint H_{u,h_P(u)}^-$ and $F = P \cap H_{u,h_P(u)}$, where $H_{u,h_P(u)}^- :=\{x \in \R^n: x \cdot u \leq h_P(u)\}$.
Thus $F \in \face{}{}([P,o]) \setminus \face{-}{}([P,o])$.
Also, let $F' \in \face{}{}([P,o]) \setminus \face{-}{}([P,o])$.
Similarly there is a plane $H_{u,h_{[P,o]}(u)}$ that does not contain the origin and $F'=[P,o] \cap H_{u,h_{[P,o]}(u)}$.
Hence $F' \in \face{}{}(P) \setminus \face{-}{}(P)$.
Thus \eqref{eq:eu4} holds.

Lemma \ref{lem:contin} shows that $ZP \in \CFo{n}$ for every $P \in \MP^n$.
Now applying Theorem \ref{mthm:sln2} to $P \in \mpo{n}$ with $\dim P<n$ and $P \in \MP^n$ with $\dim P<n-1$ (by the proof of Theorem \ref{mthm:sln} and Theorem \ref{mthm:sln2}, the assumption $ZP \in \CFo{n}$ is enough for this case), we get \eqref{eq8041} holds
for every $P \in \mpo{n}$ with $\dim P<n$ and for every $P \in \MP^n$ with $\dim P<n-1$.
By the continuity of $Z$ and Lemma \ref{lem:contin3}, we get $\zeta_2+\tilde{\zeta}_2 \equiv 0$ and $\zeta_1 \equiv 0$.

Combining \eqref{eq8041} with \eqref{eq:eu4}, \eqref{defval3} and Corollary \ref{co:eusupp}, there are suitable functions $\zeta,\tilde{\zeta} \in \CF{}$ such that
\begin{align*}
ZP(x)=\zeta(h_P(x)) +\zeta(-h_{-P}(x))+\tilde{\zeta}(h_{[P,o]}(x)) +\tilde{\zeta}(-h_{-[P,o]}(x))
\end{align*}
for every $P \in \MP^n$ if $\dim P<n-1$ and for every $P \in \mpo{n}$ if $\dim P<n$.

Define $\bar Z: \MP^n \to \FFo{n}$ by
\begin{align*}
  \bar ZP(x)=ZP(x)-(\zeta(h_P(x)) +\zeta(-h_{-P}(x))+\tilde{\zeta}(h_{[P,o]}(x)) +\tilde{\zeta}(-h_{-[P,o]}(x))).
\end{align*}
By Lemma \ref{lemuq2}, we only need to show that
\begin{align*}
\bar ZP(x)=\frac{1}{|x|} \int_{\R} V_{n-1}(P \cap H_{x,t}) d\mu(t)
+ \frac{1}{|x|} \int_{\R} V_{n-1}([P,o] \cap H_{x,t}) d\tilde{\mu}(t)
\end{align*}
for all $P=sT^n$ and $P=s[e_1,\dots,e_n]$ with $s>0$.
By Lemma \ref{lem:contin}, the remark after Theorem \ref{mthm:simpsln},
and the same idea of the proofs of Theorem \ref{mthm:sln2}, we get the desired result.
\end{proof}

Follows directly from Theorem \ref{mthm:sln4} and conclusions in \S \ref{sec:val}, we get the following.
\begin{mthm}\label{mthm:sln5}
Let $n \ge 3$. A map $Z: \MK^n \to \FFo{n}$ is a continuous and $\sln$ covariant valuation
if and only if
there are $\zeta,\tilde{\zeta} \in \CF{}$ and $\mu,\tilde{\mu} \in \cms$ such that
\begin{align*}
ZK(x)&= \zeta(h_K(x)) +\zeta(-h_{-K}(x)) + \frac{1}{|x|} \int_{\R} V_{n-1}(K \cap H_{x,t}) d\mu(t)
\\
    &\qquad + \tilde{\zeta}(h_{[K,o]}(x)) +\tilde{\zeta}(-h_{-[K,o]}(x)) + \frac{1}{|x|} \int_{\R} V_{n-1}([K,o] \cap H_{x,t}) d\tilde{\mu}(t)
\end{align*}
for every $K \in \MK^n$ and $x \in \ro{n}$.
\end{mthm}

\section{Further results}
In this section, we add some assumptions to Theorem \ref{mthm:sln3} and Theorem \ref{thm:co0} to get more specific classifications and characterizations.
Some results are new (Corollaries \ref{co:hom2}, \ref{co:hom3} and some part of Corollary \ref{co:hom1}),
and some results are known (Corollaries \ref{co:real}, \ref{co:Min}, \ref{co:harm} and some part of Corollary \ref{co:hom1}).

A map $Z: \mpo{n} \to \FFo{n}$ is $q$-homogeneous if $Z(\a P)=\a^q ZP$ for any $\a >0$ and $P \in \mpo{n}$. For $x \in \ro{n}$, the half-space $H_{x,0}^+=\{y\in\R^n: x\cdot y \ge 0\}$ and $H_{x,0}^-=\{y\in\R^n: x\cdot y \le 0\}$.
\begin{cor}\label{co:hom2}
Let $n \ge 3$ and $q \in \R$. A map $Z:\mpo{n} \to \FFo{n}$ is a continuous, $\sln$ covariant and $q$-homogeneous valuation
if and only if
there are $c_1,c_2,c_3,c_4 \in \R$ such that
\begin{align*}
&ZP(x) \\
&=\begin{cases}
  c_1 h_P(x)^q +c_2h_{-P}(x)^q + c_3 \int_{P \cap H_{x,o}^+} (x \cdot y) ^{q-n} dy + c_4 \int_{P \cap H_{x,o}^-} \ab{-x \cdot y}^{q-n} dy, & q > n-1, \\
  c_1 h_P(x)^q +c_2h_{-P}(x)^q , & 0 \leq q \leq n-1,\\
  0, &q < 0
\end{cases}
\end{align*}
for every $P \in \mpo{n}$ and $x \in \ro{n}$.
\end{cor}
\begin{proof}
Let $\zeta_1(t)$ be $c_1 t^q$ for $t \ge 0$ and $c_2 (-t)^q$ for $t<0$ ($q>0$), and the density of $\mu$ be $c_3t^{q-n}$ for $t > 0$ and $c_4 (-t)^{q-n}$ for $t<0$ ($q>n-1$) in Theorem \ref{mthm:sln3} (and \eqref{eq:fubini}).
Hence the ``if" part is true.

For the ``only if" part, by Theorem \ref{mthm:sln3}, there are $\zeta \in \CF{}$ and $\mu \in \cms$ such that
\begin{align*}
ZP(x)&= \zeta(h_P(x)) +\zeta(-h_{-P}(x)) + \frac{1}{|x|} \int_{\R} V_{n-1}(P \cap H_{x,t}) d\mu(t)
\end{align*}
for every $P \in \mpo{n}$ and $x \in \ro{n}$.
Since $Z$ is $q$-homogeneous, we have
\begin{align*}
  \zeta(h_{\a P}(x)) +\zeta(-h_{-\a P}(x))=\a^q (\zeta(h_P(x)) +\zeta(-h_{-P}(x)))
\end{align*}
for $\dim P <n$, $x \in \ro{n}$ and $\a >0$.
Further setting $P=[o,e_1]$ and $x=e_1$ or $x=-e_1$, we get
\begin{align*}
  \zeta(\a)=\a^q \zeta(1), \zeta(-\a)=\a^q \zeta(-1).
\end{align*}
Let $c_1=\zeta(1)$ and $c_2=\zeta(-1)$.
For $q \ge 0$, the condition $\zeta \in \CF{}$ gives
\begin{align*}
  ZP(x)=c_1h_P(x)^q+c_2h_{-P}(x)^q
\end{align*}
for any $\dim P <n$, $x \in \ro{n}$.
For $q<0$, the condition $\zeta \in \CF{}$ gives $c_1=0$ and $c_2=0$.

Now let
\begin{align*}
  \tilde Z P(x)=\frac{1}{|x|} \int_{\R} V_{n-1}(P \cap H_{x,t}) d\mu(t).
\end{align*}
By the above argument, $\tilde Z P(x)=ZP(x)-c_1h_P(x)^q-c_2h_{-P}(x)^q$.
Hence $\tilde Z P (\cdot) $ is also $q$-homogeneous.
Set $P=[0,1] \times [0,1]^{n-1}$ and $x=e_1$.
Since $\tilde Z (\a P)( x) = \a^q \tilde Z P(x)$, we get
\begin{align*}
  \a^{n-1} \mu(0,\a )= \a^q \mu(0,1).
\end{align*}
For arbitrary $0<r<s$, together with $\mu \in \cms$, then we have
\begin{align*}
  \mu(r,s)=\mu(0,s)-\mu(0,r)=\mu(0,1) (s^{q-n+1}-r^{q-n+1}) = \mu(0,1) (q-n+1)\int_{r}^{s} t^{q-n} dt
\end{align*}
Let $c_3=\mu(0,1) (q-n+1)$.
Since the Radon measure $\mu$ on $(0,\infty)$ is determined by its values on $(r,s)$ for all $0<r<s$ (regularity of the Radon measure), $\mu$ has a density
\begin{align*}
  \zeta(t) = c_3 t^{q-n}
\end{align*}
for $t>0$.
Further since $\mu(0,s)< \infty$ for $s>0$, we have $c_3=0$ if $q \leq n-1$.
Similarly, change $x$ to $-e_1$, then $\mu$ has a density
$\zeta(t) = c_4 t^{q-n}$
for $t<0$ with a constant $c_4 \in \R$, and we have $c_4=0$ if $q \leq n-1$.
Now the proof is completed by \eqref{eq:fubini}.
\end{proof}

We can also assume that $ZP$ is $q$-homogeneous for every $P$.
For $q \in \R$, a function $f \in \FFo{n}$ is $q$-homogeneous if $f(\alpha x)=\alpha^q f(x)$ for any $\alpha >0$.
The following Corollary was established by Liu and Wang \cite{LW2020sln} for $q >-1$ such that valuations are simple and Parapatits \cite{Par14b} for $q \ge 1$ for all continuous valuations.
The proof is quite similar to that of Corollary \ref{co:hom1} (using the same specific polytopes), so we omit the proof.

\begin{cor}\label{co:hom1}
Let $n \ge 3$ and $q \in \R$. A map $Z$, mapping $\mpo{n}$ to the space of $q$-homogeneous functions on $\ro{n}$, is a continuous and $\sln$ covariant valuation
if and only if
there are $c_1,c_2,c_3,c_4 \in \R$ such that
\begin{align*}
ZP(x)=\begin{cases}
  c_1 h_P(x)^q +c_2h_{-P}(x)^q + c_3 \int_{P \cap H_{x,o}^+} (x \cdot y) ^q dy + c_4 \int_{P \cap H_{x,o}^-} \ab{-x \cdot y}^q dy, & q \ge 0, \\
  c_3 \int_{P \cap H_{x,o}^+} (x \cdot y) ^q dy + c_4 \int_{P \cap H_{x,o}^-} \ab{-x \cdot y}^q dy, & -1< q < 0,\\
  0, &q \le -1
\end{cases}
\end{align*}
for every $P \in \mpo{n}$ and $x \in \ro{n}$.
\end{cor}

A special case of $q=0$ in Corollary \ref{co:hom1} is the classification of constant function-valued valuations, which is equivalent with the classification of $\sln$ invariant real-valued valuations.

\begin{cor}\label{co:real}
Let $n \ge 3$. A map $Z:\mpo{n} \to \R$ is a continuous and $\sln$ invariant valuation
if and only if
there are $c_1,c_2\in \R$ such that
\begin{align*}
ZP=c_1V_0(P)+c_2 V_n(P)
\end{align*}
for every $P \in \mpo{n}$.
\end{cor}
\begin{proof}
Set $\tilde Z P(x)= ZP$ for any $P \in \mpo{n}$ and $x \in \ro{n}$.
Hence $\tilde Z$ is a continuous and $\sln$ covariant valuation.
Also, $\tilde ZP$ is a $0$-homogeneous function.
By Corollary \ref{co:hom1},
\begin{align*}
  \tilde ZP(x)=a_1 +a_2 + a_3 V_n(P \cap H_{x,o}^+) + a_4 V_n(P \cap H_{x,o}^-)
\end{align*}
with some constants $a_1,a_2,a_3,a_4 \in \R$.
Choosing an asymmetric $P \in \mpo{n}$, we get $a_3=a_4$ since $\tilde ZP(x) = \tilde ZP(-x)$.
Thus we get the desired result.
\end{proof}

The function $ZP(x)=\int_P \log |x \cdot y| dy$ is not homogeneous on $x$,
but
\begin{align*}
  ZP(\a x)= ZP(x) + V(P) \log \a
\end{align*}
for any $P \in \mpo{n}$, $\a >0$ and $x \in \ro{n}$.

\begin{cor}\label{co:hom3}
Let $n \ge 3$. A map $Z:\mpo{n} \to \FFo{n}$ is a continuous and $\sln$ covariant valuation
which satisfies the following property that
\begin{align}\label{eq:exphom}
  ZP(\a x)= ZP(x) + \bar Z P(x) \log \a
\end{align}
with some $\bar Z P(x) \in \R$ for any $P \in \mpo{n}$, $\a >0$ and $x \in \ro{n}$,
if and only if
there are $c_1,c_2,c_3,c_4,c_5 \in \R$ such that
\begin{align*}
ZP(x)&=
  c_1 V_0(P) + c_2 V_n(P \cap H_{x,o}^+) + c_3 V_n(P \cap H_{x,o}^-) \\
  &\qquad
  + c_4 \int_{P \cap H_{x,o}^+} \log |x \cdot y| dy + c_5 \int_{P \cap H_{x,o}^-} \log |x \cdot y| dy
\end{align*}
for every $P \in \mpo{n}$ and $x \in \ro{n}$.
\end{cor}
\begin{proof}
Let $\zeta_1 \equiv c_1$ and the density of $\mu$ be $c_2+c_4 \log |t|$ for $t > 0$ and $c_3+c_5 \log |t|$ for $t<0$ in Theorem \ref{mthm:sln3} (and \eqref{eq:fubini}).
Also, \eqref{eq:exphom} holds for $\bar ZP(x)=c_4V_n(P \cap H_{x,o}^+) + c_5 V_n(P \cap H_{x,o}^-)$.
Hence the ``if" part is true.

Since $Z: \mpo{n} \to \FFo{n}$ is a continuous, $\sln$ covariant valuation,
the property \eqref{eq:exphom} implies directly that $\bar Z: \mpo{n} \to \FFo{n}$ is also a continuous, $\sln$ covariant valuation.
Moreover, for any $\a,\beta >0$, \eqref{eq:exphom} gives
\begin{align*}
  ZP(\a \beta x)= ZP(x) + \bar Z P(x) \log (\a \beta)
\end{align*}
and
\begin{align*}
  ZP(\a \beta x)= ZP(\a x) + \bar Z P(\a x) \log \beta = ZP(x) + \bar Z P (x) \log \a + \bar Z P(\a x) \log \beta.
\end{align*}
Hence $\bar Z P (\cdot)$ is $0$-homogeneous.
Now by Corollary \ref{co:hom1} for $q=0$, there are $b_0,b_1,b_2 \in \R$ such that
\begin{align}\label{eq:exphom2}
  \bar Z P(x)=b_0V_0(P)+b_1 V_n(P \cap H_{x,o}^+) + b_2 V_n(P \cap H_{x,o}^-).
\end{align}
for any $P \in \mpo{n}$ and $x \in \ro{n}$.

Next we apply Theorem \ref{mthm:sln3} and use a similar argument in Corollary \ref{co:hom2} to prove the ``only if" part.
By Theorem \ref{mthm:sln3}, there are $\zeta \in \CF{}$ and $\mu \in \cms$ such that
\begin{align*}
ZP(x)&= \zeta(h_P(x)) +\zeta(-h_{-P}(x)) + \frac{1}{|x|} \int_{\R} V_{n-1}(P \cap H_{x,t}) d\mu(t)
\end{align*}
Consider $P=[o,e_1]$ and $x=e_1$ or $x=-e_1$.
By \eqref{eq:exphom} and \eqref{eq:exphom2}, we get
\begin{align*}
  \zeta(\a)=\zeta(1) + b_0 \log \a,~\zeta(-\a)=\zeta(-1)+b_0 \log \a
\end{align*}
for any $\a >0$.
Recall that $\zeta \in \CF{}$.
Thus $b_0=0$ and $\zeta(\a)=\zeta(1)$ for any $\a \in \R$.
Denoted by $c_1=\zeta(1)$, we get
\begin{align*}
  ZP(x)=c_1V_0(P)
\end{align*}
for every $P \in \mpo{n}$ with $\dim P<n$ and $x \in \ro{n}$.

For the full-dimensional case, we set
\begin{align*}
  \tilde Z P(x)=\frac{1}{|x|} \int_{\R} V_{n-1}(P \cap H_{x,t}) d \mu(t),
\end{align*}
which is equal to $ZP(x)-c_1V_0(P)$ by the above argument.
Set $P=[0,1] \times [0,1]^{n-1}$ and $x=e_1$ or $x=-e_1$.
By \eqref{eq:exphom} and \eqref{eq:exphom2}, we have
\begin{align*}
  \a^{-1} \mu(0,\a )= \mu(0,1) + b_1 \log \a,~~
  \a^{-1} \mu(-\a,0 )= \mu(-1,0) + b_2 \log \a.
\end{align*}
For arbitrary $0<r<s$, together with $\mu \in \cms$, we have
\begin{align*}
  \mu(r,s)=\mu(0,s)-\mu(0,r)&=\mu(0,1)(s-r) + b_1(s\log s- r \log r) \\
   &= \int_{r}^{s} \mu(0,1) + b_1 + b_1 \log t ~ dt
\end{align*}
and
\begin{align*}
  \mu(-s,-r)=\mu(-s,0)-\mu(-r,0)&=\mu(-1,0)(s-r) + b_2(s\log s- r \log r)\\
  &=\int_{-s}^{-r} \mu(-1,0) + b_2 + b_2 \log (-t) ~dt
\end{align*}

Now set $c_2=\mu(0,1) + b_1 $, $c_3=\mu(-1,0) + b_2$, $c_4=b_1$ and $c_5=b_2$, then the proof is completed by \eqref{eq:fubini}.
\end{proof}

Corollary \ref{co:hom1} has special geometric interests since it can be associated with some geometric additions. The following two corollaries can be also found in \cite{Hab12b,Par14b,LW2020sln}.

Let $\MK_c^n$ be the set of origin-symmetric convex bodies in $\R^n$.
A map $Z: \mpo{n} \to \MK_c^n$ is $\sln$ covariant if $Z(\phi P) = \phi ZP$ for any $\phi \in \sln$ and $P \in \mpo{n}$.

\begin{cor}\label{co:Min}
Let $n \ge 3$ and $q \ge 1$. A map $Z: \mpo{n} \to \MK_c^n$ is a continuous and $\sln$ covariant $L_q$ Minkowski valuation
if and only if
there are $c_1,c_2 \in \R$ such that
\begin{align*}
ZP(x)=
  c_1 (P +_q (-P)) +_q c_2 M_{q} P
\end{align*}
for every $P \in \mpo{n}$.
\end{cor}
\begin{proof}
Set $\tilde Z: \mpo{n} \to \FFo{n}$ by
$\tilde Z P(x)=h_{ZP}(x)^q$
for any $P \in \mpo{n}$ and $x\in \ro{n}$.
It is easy to see that $\tilde{Z}$ is a continuous, $\sln$ covariant valuation on $\mpo{n}$ which takes values at $q$-homogeneous functions.
Consider an asymmetric $P \in \mpo{n}$ for $\dim P <n$ and for $\dim P=n$ separately.  Together with \eqref{def:pmom}, the desired result follows directly from Corollary \ref{co:hom1}.
\end{proof}

Let $\MS_c^n$ denote the space of origin-symmetric, compact, star-shaped sets which contain the origin in their interior.
For $q \neq 0$, the $L_q$ harmonic sum $K \widehat{+}_q L$ of $K,L \in \MS_c^n$ is
\begin{align*}
  \|x\|_{K \widehat{+}_q L}=\left(\|x\|_{K}^q + \|x\|_{L}^q \right)^{1/q}.
\end{align*}
For $q \ge 1$, $K^\ast \widehat{+}_q L^\ast =(K +_q L)^\ast$.

We use the notation $\langle \MS^n_c,\widehat{+}_q \rangle$ as in \cite{Hab09,LW2020sln} to denote
the smallest monoid containing $\MS^n_c$ with respect to the $L_q$ harmonic addition.
Thus $\langle \MS^n_c,\widehat{+}_q \rangle$ contains the set $\R^n$ for $q >0$ and $\{o\}$ for $q <0$, respectively.
A map $Z: \mpo{n} \to \langle \MS^n_c,\widehat{+}_q \rangle$ is $\sln$ contravariant if $Z(\phi P) = \phi^{-t} ZP$ for any $\phi \in \sln$ and $P \in \mpo{n}$.

\begin{cor}\label{co:harm}
Let $n \ge 3$ and $q \neq 0$. A map $Z: \mpo{n} \to \langle \MS^n_c,\widehat{+}_q \rangle$ is a continuous and $\sln$ contravariant valuation
if and only if
there is a constant $c_1 \in \R$ such that
\begin{align*}
ZP(x)=\begin{cases}
   c_1 M_{q}^\ast P, & q > -1, \\
   \{o\}, & q\le -1.
\end{cases}
\end{align*}
for every $P \in \mpo{n}$.
\end{cor}
\begin{proof}
For $q \neq 0$, set $\tilde Z: \mpo{n} \to \FFo{n}$ by
  $\tilde Z P(x)=\|x\|_{ZP}^q$
for any $P \in \mpo{n}$ and $x\in \ro{n}$.
It is easy to see that $\tilde Z$
is a continuous, $\sln$ contravariant valuation on $\mpo{n}$ which takes values at $q$-homogeneous functions.
By Corollary \ref{co:hom1}, we have
\begin{align*}
\|x\|_{ZP}^q =\begin{cases}
  c_1 h_P(x)^q +c_2h_{-P}(x)^q + c_3 \int_{P \cap H_{x,o}^+} (x \cdot y) ^q dy + c_4 \int_{P \cap H_{x,o}^-} \ab{-x \cdot y}^q dy, & q > 0, \\
  c_3 \int_{P \cap H_{x,o}^+} (x \cdot y) ^q dy + c_4 \int_{P \cap H_{x,o}^-} \ab{-x \cdot y}^q dy, & -1< q < 0,\\
  0, &q \le -1.
\end{cases}
\end{align*}
for $P \in \mpo{n}$ and $x \in \ro{n}$.
Let first $\dim P <n$.
By the definition of $\langle \MS^n_c,\widehat{+}_q \rangle$, either $\|x\|_{ZP}^q \equiv 0$ or $\|x\|_{ZP}>0$ for all $x \in \ro{n}$.
Hence $c_1=c_2=0$  (e.g. consider $P=[o,e_1]$).
Further consider $\dim P=n$ and an asymmetric $P$, then $c_3=c_4$.
Now by \eqref{def:polarmom}, we get the desired result for $q \neq 0$.
(For $q \leq -1$, $\|\cdot\|_{ZP}^q \equiv 0$ means $\|\cdot\|_{ZP} \equiv \infty$ and hence $ZP = \{o\}$.)
\end{proof}

\section*{Acknowledgement}
\addcontentsline{toc}{section}{Acknowledgement}
The work of the author was supported in part by the Austrian Science Fund (FWF M2642 and I3027), the European Research Council (ERC 306445), and the National Natural Science Foundation of China (11671249).

\end{document}